\documentclass[review]{elsarticle}

\usepackage{lineno,hyperref}
\modulolinenumbers[5]

\journal{Physica D: Nonlinear Phenomena}

\usepackage{dsfont} 
\usepackage{amssymb} 
\usepackage{mathrsfs} 
\usepackage{amsthm}
\usepackage{amsmath}
\usepackage{graphicx}
\usepackage{longtable}
\usepackage{graphicx}
\usepackage{subfigure}				
\usepackage{txfonts}									
\usepackage{amssymb}
\usepackage{changepage}
\usepackage{mathrsfs} 
\usepackage{color}

\begin{document}

\newtheorem{definition}{Definition}[section]
\newtheorem{lemma}{Lemma}[section]
\newtheorem{remark}{Remark}[section]
\newtheorem{theorem}{Theorem}[section]
\newtheorem{example}{Example}[section]
\newtheorem{corollary}{Corollary}[section]

\begin{frontmatter}

\title{Hamiltonian Systems with L\'evy Noise: Symplecticity, Hamilton's Principle and Averaging Principle\tnoteref{mytitlenote}}
\tnotetext[mytitlenote]{This work was partly supported by the NSF grant 1620449,  and NSFC grants 11531006 and 11771449.}


\author[mymainaddress]{Pingyuan Wei}
\ead{weipingyuan@hust.edu.cn}

\author[mymainaddress]{Ying Chao}
\ead{yingchao1993@hust.edu.cn}

\author[mymainaddress,mysecondaryaddress]{Jinqiao Duan\corref{mycorrespondingauthor}}
\cortext[mycorrespondingauthor]{Corresponding author}
\ead{duan@iit.edu}

\address[mymainaddress]{School of Mathematics and Statistics, \& Center for Mathematical Sciences,  Huazhong University of Sciences and Technology, Wuhan 430074,  China}
\address[mysecondaryaddress]{Department of Applied Mathematics, Illinois Institute of Technology, Chicago, IL 60616, USA}

\begin{abstract}
This work focuses on topics related to Hamiltonian stochastic differential equations with L\'{e}vy noise. We first show that the phase flow of the  stochastic system preserves symplectic structure, and propose a stochastic version of Hamilton's principle by the corresponding formulation of the stochastic action integral and the Euler-Lagrange equation. Based on these properties, we further investigate the effective behaviour of a small transversal perturbation to a completely integrable stochastic Hamiltonian system with L\'{e}vy noise. We establish an averaging principle in the sense that the action component of solution converges to the solution of a stochastic differential equation when the scale parameter goes to zero. Furthermore, we obtain the estimation for the rate of this convergence. Finally, we present an example to illustrate these results.
\end{abstract}

\begin{keyword}
 Stochastic Hamiltonian systems; L\'evy noise; symplecticity; Hamilton's principle; averaging principle.
\MSC[2010] 60H10 \sep  60J51 \sep 58J37
\end{keyword}

\end{frontmatter}


\renewcommand{\theequation}{\thesection.\arabic{equation}}
\setcounter{equation}{0}

\section{Introduction }
\noindent
Certain nonlinear systems  have ``geometric'' structures, such as the Hamiltonian  structure \cite{Ar,Me,Wig}. Hamiltonian systems of ordinary differential equations (ODEs) widely appear in celestial mechanics, statistical mechanics, geophysics, and chemical physics. They are models for the dynamics of planets, motion of particles in a fluid,  and evolution of    other microscopic systems  \cite{Reza}. Hamiltonian systems have many  well-known properties. For example, it was known to Liouville that the flows of Hamiltonian systems possess the property of phase-volume preservation; Poincar\'e observed that the Hamiltonian flows are symplectic and geometrically preserve certain symplectic area along phase flow \cite{Poin}; based on Hamilton's principle, Hamiltonian equations of motion are closely   related to Euler-Lagrange differential equations \cite{Me,Tve}. As a matter of fact, these dynamical systems are often subject to perturbations. In the deterministic case, the perturbation theory of Hamiltonian systems have appeared long ago; see Arnold \cite{Ar} and Freidlin-Wentzell \cite{Fr} for details. Particularly, an averaging principle for an integrable Hamiltonian system has been studied in e.g. Arnold \cite{Ar}.
\par
It is important to take randomness into account when building mathematical models for complex phenomena  under uncertainty \cite{Duan}.
Stochastic differential equations (SDEs) with ``Hamiltonian structures" are appropriate models for randomly influenced Hamiltonian systems as studied in Bismut \cite{Bi}, and have also drawn much attention; see, for example, Brin-Freidlin \cite{BMFM}, MacKay \cite{Mac}, Misawa \cite{Misa}, Wu \cite{Wu}, Zhu-Huang \cite{Zhu}. In particular, Milstein et al. \cite{Mi2, Mi} proved the symplecticity for stochastic Hamiltonian systems with Brownian noise, and Wang et al. \cite{Wang} proposed a version of Hamilton's principle for the same systems to construct variational integrators; Pavon \cite{Pav} established variational principles in stochastic mechanics; Li \cite{Li} developed an averaging principle for a perturbed completely integrable stochastic Hamiltonian system with Brownian noise. For some specific physical Hamiltonian models, we refer to Cresson-Darses \cite{Cr} and Givon et al. \cite{Gi}.
\par
In view of the development on SDEs with Hamiltonian structures, the noise processes considered to date are mainly Gaussian noise in terms of Brownian motion. However, non-Gaussian random fluctuations should be introduced to capture some large moves and unpredictable events in various areas such as not only aforementioned celestial mechanics and statistical physics, but also mathematics finance and life science \cite{Duan,Ap,Sato,GarL}. L\'{e}vy motions are an important and useful class of non-Gaussian processes whose sample paths are  c\`{a}dl\`{a}g (right-continuous with left limit at each time instant). The study on stochastic systems driven by such processes have received increasing attentions recently, especially on developing proper averaging principles for these systems. For example, Albeverio et al. \cite{Br, Al} established ergodicity of L\'{e}vy-type operators and SDEs driven by jump noise with non-Lipschitz coefficients; H\"ogele-Ruffino \cite{Hog} and Gargate-Ruffino \cite{GGI} focused on averaging along foliated Brownian and L\'{e}vy diffusions, respectively, which generalized the approach by Li \cite{Li}, and H\"ogele-da Costa \cite{Hog2} further studied strong averaging along foliated L\'{e}vy diffusions with heavy tails on compact leaves.  For more information on averaging principle for stochastic systems driven by L\'{e}vy noise, we refer to Xu et al. \cite{Xu} and Bao et al. \cite{Bao}. ODEs and SDEs with ``Hamiltonian structures" usually exhibit some extraordinary properties. Nevertheless, averaging principles for SDEs driven by L\'{e}vy noise with ``Hamiltonian structures", and even some basic dynamics such as symplecticity (invariance under a transformation) and L\'{e}vy-type stochastic Hamilton's principle of these systems, have not yet been considered to date to the best of our knowledge.
\par
In this present paper,  we consider stochastic Hamiltonian systems with L\'{e}vy noise on symplectic manifolds. They are defined as Marcus SDEs whose drift vector fields and diffusion vector fields are Hamiltonian vector fields. Note that the Marcus integral  \cite{Ku,KPP, Mar} in L\'{e}vy case has the advantage of leading to ordinary chain rule of the Newton-Leibniz type under a transformation. This property makes the Marcus integral natural to use especially in connection with SDEs on manifolds \cite{El}.
\par
We first demonstrate that the phase flow of a stochastic Hamiltonian system with L\'{e}vy noise preserves symplectic structure, and then propose the formulation of L\'{e}vy-type stochastic action integral and Euler-Lagrange equation of motions, as well as the stochastic Hamilton's principle. These properties are derived by using the calculus of variations, and the demand of the systems being in Marcus sense will simplify the stochastic differential calculations in the proofs. It is important to note that the stochastic Hamiltonian systems with L\'{e}vy noise should be understood as special nonconservative systems, for which the L\'{e}vy noise is a nonconservative `force'. The symplecticity here is presented for the whole stochastic system instead of the original deterministic Hamiltonian system without the nonconservative force. The stochastic Hamilton's principle is also proposed on the basis of nonconservative mechanical systems.
\par
Based on these foundational work, we further investigate the effective behavior of a small transversal perturbation to a (completely) integrable stochastic Hamiltonian system with L\'{e}vy noise. As this integrable stochastic system is perturbed by a transversal smooth vector field of order $\varepsilon$ ($\varepsilon$ is a small parameter), the solution to the perturbed equation will not preserve the properties mentioned above. The main idea we will use is to consider the solution along the rescaled time $t/\varepsilon$. The motion splits into two parts with fast rotation along the unperturbed trajectories and slow motion across them. Indeed, by an action-angle coordinate, the fast rotation is an diffusion on the invariant torus and the slow motion is governed by the  transversal component. When averaged by ergodic invariant measure on torus, the evolution of action component of the motion does not depend on the angular variable when $\varepsilon$ tends to zero. The essential transversal behavior is captured by a system of  ODEs for the transversal component and this result is referred as an averaging principle. The estimation for rate of convergence for this averaing principle is also established. Some inspiration for this part came from Li \cite{Li}, as well as H\"ogele-de Costa \cite{Hog2}. The main novelty of our work is that the model we consider here combines features of a Hamiltonian structure with stochastic non-Gaussian L\'evy noise.
\par
The rest of this paper is organized as follows.
In Section 2, we recall basic concepts about Hamiltonian vector fields and L\'{e}vy motions, and then present the definition of stochastic Hamiltonian system with L\'{e}vy noise, together with  the existence and uniqueness of  the solution.
In Section 3, we show that the phase flow of this stochastic system preserves the symplectic structure. By considering a stochastic Hamiltonian system with L\'{e}vy noise as a special nonconservative system, we propose a stochastic version of Hamilton's principle. The goal of this section is to better understand such a system and to establish foundation for the following sections of this paper. In Section 4, we investigate an integrable stochastic Hamiltonian system,  with L\'{e}vy noise,  perturbed by a transversal smooth vector field. After discussing the ergodic behavior and some technical issues, we  establish an averaging principle, together with a specific illustrative example.

\renewcommand{\theequation}{\thesection.\arabic{equation}}
\setcounter{equation}{0}

\section{Preliminaries}
\subsection{Stochastic Hamiltonian Systems with L\'evy Noise}
\noindent
Let $( \Omega , \mathscr{F}, \{ \mathscr{F}_t \}_{t\geqslant0} , P)$ be a filtered probability space endowed with a Poisson random measure $N$ on $( \mathbb{R}^d\setminus \{ 0 \})\times \mathbb{R}_{+}$ with jump intensity measure $\nu = \mathbb{E}N(1,\cdot)$. Denote by $\widetilde N$ the associated compensated Poisson random measure, that is, $\widetilde N(dt,dz) = N(dt,dz) - \nu (dz)dt$. We assume that the filtration $\{ \mathscr{F}_t)_{t\geqslant0} \}$ satisfies the usual conditions \cite{Protter}. Let $L_t={L(t)}$ be a $d$-dimensional  L\'{e}vy process with the generating triplet $(\gamma, A, \nu)$. By L\'evy-It\^o decomposition \cite{Ap, Sato, Ku}, 
$$
{L_t}= \gamma t + B_{A}(t) + \int_{|z|< 1} z \widetilde N(t,dz) + \int_{|z|\ge 1} z N(t,dz),
$$
where $\gamma \in {\mathbb{R}^d}$ is a drift vector, $B_A(t)$ is an independent $d$-dimensional Brownian motion with covariance matrix $A$, and the last two terms describe the `small jumps' and `big jumps' of L\'{e}vy process, respectively. In the following, we denote $L_c(t)=\gamma t + B_{A}(t)$ as the continuous part of $L_t$ and $L_d(t)=L_t-L_c(t)$ as the discontinuous part.
\par
Given a smooth Hamiltonian $H_0$ and a family of $n$ smooth Hamiltonians $\{ H_k \}_{k=1}^{n}$ on a smooth $2n$-dimensional manifold $M$ \cite{Ar,AM}. We denote by $V_0$ and $V_k$ $(k=1,2,...,d)$ the corresponding Hamiltonian vector fields, that is, 
$$dH_0(\cdot)=\omega^2(\cdot,V_0), \; dH_k(\cdot)=\omega^2(\cdot,V_k),$$
where $\omega^2$ is the symplectic form. Note that we use the symbol with superscript $2$ for the symplectic form to avoid confusion with the customary symbol for chance variable on sample space $\Omega$. \\
\par
We shall consider stochastic Hamiltonian systems driven by non-Gaussian L\'evy noise, which are described by the following SDEs in the Marcus form on $M$:
\begin{equation}\label{Equation-1}
dX=V_0(X)dt+\sum_{k=1}^d V_k (X) \diamond dL^k (t), ~ X_0:=X(t_0)=x\in M,
\end{equation}
or equivalently,
\begin{equation}\label{Equation-2}
X_t=x+\int_0^t V_0(X_s)ds+\sum_{k=1}^d \int_0^t V_k (X_{s-}) \diamond dL^k (s),
\end{equation}
where ``$\diamond$" stands for Marcus integral \cite{Ku, KPP, Mar} defined by
\begin{eqnarray}\label{Marcus}
\int_0^t V_k (X_{s-}) \diamond dL_k (s)
=\int_0^t V_k (X_{s-}) \circ dL_c^k(s)+\int_0^t V_k (X_{s-})dL_d^k(s)\notag\\
+\sum_{0\leqslant s\leqslant t}\left[ \phi(\Delta L^k (s), V_k(X_{s-}), X_{s-}) - X_{s-}-V_k (X_{s-})\Delta L^k (s)
\right]
\end{eqnarray}
with $\int\circ dL_c^k(s)$ denoting the Stratonovtich integral, $\int dL_d^k(s)$ denoting the It\^o integral and
$\phi(l,v(x),x)$ being the value at $t=1$ of the solution of the following ODE:
\begin{equation}\label{Equation-phi}
\frac{d}{dt}\xi(t)= v(\xi(t))l,~\xi(0)=x.
\end{equation}
Note that Marcus SDEs (\ref{Equation-1}) satisfy chain rule under a transformation (change of variable) and $\mathbb{P}(X_0\in M)=1$ implies that $\mathbb{P}(X_t\in M, t\geqslant 0)=1$, for details see Kurtz et al. \cite{KPP}.
\par
We remark that, by L\'evy-It\^o decomposition, the systems (\ref{Equation-1}) with L\'evy triplet being $(0,I,0)$ are stochastic Hamiltonian systems with Brownian noise \cite{Mi2,Mi,Wang}, and the system (\ref{Equation-1}) without L\'evy term are deterministic Hamiltonian systems. 


\subsection{ Existence and Uniqueness} 
In order to ensure the existence and uniqueness for the stochastic dynamical systems with Hamiltonian structure, we will need to make some assumptions. First we rewrite the Marcus equations (\ref{Equation-1}) and (\ref{Equation-2}) in the It\^o form \cite{Ap,Ku}. This can be carried out by employing the L\'evy-It\^o decomposition. Note that it's convenient to write the $d$-dimensional Brownian term in the form: $B_A(t)=\sigma B(t)$ \cite{Duan, Ap},  where $B(t)$ is a $d^{\prime}$-dimensional standard Brownian motion and $\sigma$ is a $d \times d^{\prime}$ nonzero matrix for which $A=\sigma\sigma^T$. For simplicity, we consider the Brownian term as a standard Brownian motion here, i.e., we set $A=I$. Then we obtain, for $1\leqslant i \leqslant 2n$, $t\geqslant0$,
\begin{eqnarray}\label{Eqnarray-3}
dX_t^i
&= &V_0^{i}(X_t)dt+\sum_{k=1}^d \gamma^k V_{k}^{i}(X_t)dt+\sum_{k=1}^d V_{k}^{i}(X_t)dB^k(t) + \frac{1}{2}\sum_{k=1}^d V_{k}\cdot \nabla V_{k}^{i}(X_t)dt \notag\\
&&+\int_{|z|<1}[\phi^{i}(z)(X_{t-})-X_{t-}^i]\tilde{N}(dt,dz) +\int_{|z|\geqslant1}[\phi^{i}(z)(X_{t-})-X_{t-}^i]{N}(dt,dz) \notag\\
&&+\int_{|z|<1}[\phi^{i}(z)(X_{t-})-X_t^i-\sum_{k=1}^d z^k V_k^i(X_{t-})]\nu(dz)dt.
\end{eqnarray}

\par
Denote by $\hat{D}V(x)$ the vector in $M$ whose $i$-th component is ${\max_{1\leqslant k \leqslant d} |V_{k}\cdot \nabla V_{k}^{i}(x)|}$ for $1\leqslant i \leqslant 2n, x\in M$. We make the following assumptions.
 \emph{
\begin{itemize}
\item[\bf{A1}.] The vector field $V_0$ is locally Lipschitz and the  vector fields $V_k$ $(k=1,2,...,d)$ and $\hat{D}V(x)$ are globally Lipschitz in the following sense:\\
(i) For any $x\in M$, there exists a neighborhood $M_0$ of $x$ such that $V_0|_{M_0}$  is Lipschitz continuous, i.e. there is a constant $N_1(M_0)>0$ such that, 
$$
|V_0(x_1)-V_0(x_2)|\leqslant  N_1|x_1-x_2|,~x_1, x_2\in M_0.
$$
(ii) There is a constant $N_2>0$ such that, 
$$
\max_{1\leqslant k \leqslant d} |V_k(x_1)-V_k(x_2)|^2+|\hat{D}V(x_1)-\hat{D}V(x_2)|^2\leq N_1|x_1-x_2|^2, ~x_1, x_2\in M.
$$
\item[\bf {A2}.] One sided linear growth condition: There exists a constant $N_3>0$ such that 
$$
\sum_{k=1}^d V_{k}^2(x)+2x\cdot V_0(x)\leq N_3(1+|x|^2), ~ x \in M.
$$
\end{itemize}
}
\begin{theorem}
Under assumptions {\bf A1} and {\bf A2}, there exists a unique global solution to (\ref{Eqnarray-3}), and the solution process is adapted and c\`{a}dl\`{a}g.
\end{theorem}
\begin{proof} 
This follows immediately form \cite[Theorem 3.1]{Br} and \cite[Lemma 6.10.3]{Ap}; see also \cite{Xi}.  
\end{proof}

\renewcommand{\theequation}{\thesection.\arabic{equation}}
\setcounter{equation}{0}

\section{Symplecticity and stochastic Hamilton's principle}
\noindent
In this section we present several facts about the stochastic Hamiltonian system with L\'{e}vy noise, such as the property of preserving symplectic structure and stochastic Hamilton's priciple, which will help us to better understand such systems from the viewpoint of geometry and physics and further allow us in the next sections to confine our studies to its special structure.


\subsection{Preservation of symplectic structure} \label{Section 3.1}
\noindent
Phase flows of both deterministic Hamiltonian systems and stochastic Hamiltonian systems with Brownian noise are known to preserve symplectic structure \cite{Ar, Bi, Poin}. We next show that stochastic Hamiltonian systems with L\'{e}vy noise in the Marcus sense also have this intrinsic property.\\
\par
Keeping in mind that Marcus integral satisfies the change of variable formula \cite[Section 4]{KPP}, for simplicity, we rewrite systems (\ref{Equation-1}) in their canonical coordinates. That is, with $X=(Q,P)$, $X_0=(q,p)$, $V = (\frac{\partial H}{\partial P},-\frac{\partial H}{\partial Q})$ and $V_k = (\frac{\partial H_k}{\partial P},-\frac{\partial H_k}{\partial Q})$, $k=1,...,d$, canonical stochastic Hamiltonian systems with L\'{e}vy noise are
\begin{equation}\label{Equation-q1}
dQ=\frac{\partial H}{\partial P}(Q,P)dt+\sum_{k=1}^d \frac{\partial H_k}{\partial P}(Q,P) \diamond dL^k (t), ~ Q(t_0)=q,
\end{equation}
\begin{equation}\label{Equation-p1}
dP=-\frac{\partial H}{\partial Q}(Q,P)dt-\sum_{k=1}^d \frac{\partial H_k}{\partial Q}(Q,P) \diamond dL^k (t), ~ P(t_0)=p.
\end{equation}
\par
Note that $dp\wedge dq=\sum_{i=1}^ndp_i \wedge dq_i$ determines a differential two-form.
We are interested in systems (\ref{Equation-q1} - \ref{Equation-p1}) such that the transformation $(p, q) \to (P, Q)$ preserves symplectic structure as follows:
\begin{align}\label{Equation-qp}
dP\wedge dQ &=dp\wedge dq, \notag\\
i.e.,  \;  \sum_{i=1}^ndP_i \wedge dQ_i &=\sum_{i=1}^ndp_i \wedge dq_i.
\end{align}
To avoid confusion, we should note that the differentials in (\ref{Equation-q1})-(\ref{Equation-p1}) and (\ref{Equation-qp}) have different meanings: In (\ref{Equation-q1})-(\ref{Equation-p1}), $P, Q$ are treated as functions of time and $p, q$ are fixed parameters, while, in (\ref{Equation-qp}), the differentiation is made with respect to the initial data $p, q$ . 
\par
Geometrically, (\ref{Equation-qp}) means that the sum of the oriented areas of projections is an integral invariant \cite{Ar, Mi}. Consequently, for such systems, all exterior powers of the two-form are also invariant, and the case of $n$-th exterior power gives the preservation of phase volume.

\begin{theorem}\label{theorem 3.1}  (Symplecticity)
The stochastic Hamiltonian system (\ref{Equation-q1} - \ref{Equation-p1}) preserves symplectic structure.
\end{theorem}
The proof of this theorem is based on the differential transformation in the sense of Marcus. It is given in the Appendix.


\subsection{Stochastic Hamilton's Principle with L\'{e}vy noise} 
\noindent
For conservative mechanical systems, the classicical Hamilton's principle asserts that the dynamics of systems are determined by a variational problem for Lagrangian, and it gives a relationship between the Euler-Lagrange equation and the action integral of the motion \cite{Ar}. For the situation of nonconservative mechanical systems, the form of the action integral and that of the Euler-Lagrange equation must be changed \cite{Tve,Wang}. In this subsection, we would like to propose a stochastic version of Hamilton's principle for a stochastic Hamiltonian system with L\'{e}vy noise by viewing it as a special nonconservative system.\\
\par
We recall some results of nonconservative mechanical systems at first. Let $\mathbf{ F}$ be a nonconservative generalized force. The work done by this nonconservative generalized force is defined as
\begin{equation} \label{Equation-work}
W=-\mathbf{ F}\cdot \mathbf{r},
\end{equation}
where $\mathbf{ r}=\mathbf{ r}(q,t)$ being a position vector. As a nonconservative generalized force is independent of generalized configuration $q$, the variation of $W$ satisfies
$$
\delta W=\mathbf{ F}\cdot\delta \mathbf{ r}=\mathbf{ F}\cdot \frac{\partial\mathbf{ r}}{\partial q}\delta q.
$$
\par
Let $L(q,\dot{q},t)$ be a Lagrangian with respect to original conservative Hamiltonian system, and it is connected with Hamiltonian $H$ through the equation
\begin{equation}\label{Legendre}
L=p\cdot \dot{q}-H,
\end{equation}
where $p= \frac{\partial L}{\partial \dot{q}}$ is the Legendre transform. Consider $\gamma=\{ q(t): t_0\leqslant t \leqslant t_1\}$ as a temporally parameterized curve in the configuration space. Under the influence of $\mathbf{ F}$, the action integral of this curve is defined by
\begin{equation}\label{Equation-action integral}
\mathcal{S}[\gamma]=\int_{t_0}^{t_1}(L(\gamma(t),\dot{\gamma}(t),t)-W(\gamma(t)))dt.
\end{equation}
Hamilton's principle of nonconservative mechanical systems asserts that $\delta\mathcal{S}=0$ is equal to the following Euler-Lagrange equation holds:
\begin{equation}\label{Equation-Lagrange}
\frac{d}{dt} \frac{\partial L}{\partial \dot{q}} - \frac{\partial L}{\partial q} = \mathbf{ F}\cdot \frac{\partial\mathbf{ r}}{\partial q}.
\end{equation}
Here the Lagrangian $L$ is considered as a function with independent variables $q$, $\dot{q}$ and $t$.
\par

It is known to \cite{Tve} that the Euler-Lagrange equations of motion have the property of redundancy. As the value of Lagrangian is invariant to variable transformations, Lagrangian $L$ can be transformed from the variable set $\{ q \}$ to a redundant variable set $\{ Q^{\ast}, P^{\ast} \} $ by
$$
L(q,\dot{q},t)=L(q(Q^{\ast},P^{\ast},t),\dot{q}(Q^{\ast},P^{\ast},\dot{Q^{\ast}},\dot{P^{\ast}},t),t)=L(Q^{\ast},P^{\ast},\dot{Q^{\ast}},\dot{P^{\ast}},t).
$$
 With generalized independent variables $Q^{\ast}$, $P^{\ast}$, $\dot{Q^{\ast}}$, $\dot{P^{\ast}}$ and $t$, the generalized Euler-Lagrange equations of motion can be represented as,
\begin{eqnarray}
\frac{d}{dt} \frac{\partial L}{\partial \dot{P^{\ast}}} - \frac{\partial L}{\partial P^{\ast}} &=& \mathbf{ F}\cdot \frac{\partial\mathbf{ r}}{\partial P^{\ast}},  \label{Eqnarray-Lagrange1} \\
\frac{d}{dt} \frac{\partial L}{\partial \dot{Q^{\ast}}} - \frac{\partial L}{\partial Q^{\ast}} &=& \mathbf{ F}\cdot \frac{\partial\mathbf{ r}}{\partial Q^{\ast}}  \label{Eqnarray-Lagrange2}
\end{eqnarray}
with the position vector $\mathbf{ r} = \mathbf{ r}(Q^{\ast},P^{\ast},t)$. Based on (\ref{Eqnarray-Lagrange1} - \ref{Eqnarray-Lagrange2}), for a nonconservative system with nonconservative force $\mathbf{ F}$, the corresponding generalized Hamiltonian equations take the following form \cite{Tve}
\begin{eqnarray}
\dot{Q^{\ast}}&=&\frac{\partial H}{\partial P^{\ast}}- \frac{\partial \mathbf{ r}}{\partial P^{\ast}}\cdot\mathbf{ F},  \label{Eqnarray-gH1} \\
\dot{P^{\ast}}&=&-\frac{\partial H}{\partial Q^{\ast}}+ \frac{\partial \mathbf{ r}}{\partial Q^{\ast}}\cdot\mathbf{ F}.  \label{Eqnarray-gH2}
\end{eqnarray}
\\
\indent L\'{e}vy noise as a kind of random fluctuating force, can be treated as a special nonconservative force \cite{Ap, Zhu}. We rewrite a stochastic Hamiltonian system with L\'{e}vy noise (\ref{Equation-q1} - \ref{Equation-p1}) in the following form
\begin{eqnarray}
\dot{Q}&=&\frac{\partial H}{\partial P} + \frac{\partial \bar{H}}{\partial P}\diamond\dot{L}(t),  \label{Eqnarray-Q-noise} \\
\dot{P}&=&-\frac{\partial H}{\partial Q} - \frac{\partial \bar{H}}{\partial Q}\diamond\dot{L}(t).  \label{Eqnarray-P-noise}
\end{eqnarray}
where $\bar{H}=(H_1,H_2,...,H_d)$. 
It is natural to compare (\ref{Eqnarray-gH1} - \ref{Eqnarray-gH2}) with (\ref{Eqnarray-Q-noise} - \ref{Eqnarray-P-noise}). Formally, the associations between $\mathbf{ F}$ and $\dot{L}(t)$, as well as $ \mathbf{ r}$ and $-\bar{H} $ are reasonable. Under this consideration, we can thus view stochastic Hamiltonian systems with L\'{e}vy noise as a special class of nonconservative system. In other words, stochastic Hamiltonian systems  with L\'{e}vy noise are Hamiltonian systems in certain generalized sense, which are disturbed by certain nonconservative force (i.e., L\'{e}vy noise).

It should be noted that the random fluctuating force here, i.e. L\'{e}vy noise, is different from usual nonconservative forces which dissipate energy of the system. L\'{e}vy noise may also `add' energy to the system. To  illustrate this point, we consider the following linear stochastic oscillator.

\begin{example}\label{example 3.1} (Linear stochastic oscillator with L\'{e}vy noise)
\begin{eqnarray}
dx&=&ydt, ~~~~~~~~~~~~~~ x(t_0)=x_0,\label{Linear SO1}\\
dy&=&-xdt-\sigma dL_t, ~~ y(t_0)=y_0.\label{Linear SO2}
\end{eqnarray}
which is a stochastic Hamiltonian system with $H(x,y)=\frac{1}{2}(x^2+y^2)$ and $H_1(x,y)=\sigma y$ ($\sigma>0$ is a constant). Rewrite it in 2-dimensional vector form and multiply both sides with the integrating factor
$e^{tJ}$, where $\scriptsize{J=}\begin{bmatrix}\begin{smallmatrix}0 & 1 \\ -1 &  0\end{smallmatrix}\end{bmatrix}$. It's not hard to show that this equation has the unique solution
\begin{eqnarray}
x(t)&=&x(0)\cos t+y(0)\sin t+\int_0^t \sigma\sin(t-s) dL_s, \label{Linear SOsln1}\\
y(t)&=&-x(0)\sin t+y(0)\cos t+\int_0^t \sigma\cos(t-s) dL_s. \label{Linear SOsln2}
\end{eqnarray}
\par
For simplicity, we take the initial conditions $x_0=1$, $y_0=0$ and the drift of L\'evy motion $\gamma=0$. In the sense of L\'evy-It\^o decomposition, solution (\ref{Linear SOsln1} - \ref{Linear SOsln2}) involves a `large jumps' term. By using interlacing \cite[Page 365]{Ap}, it makes sense to begin by omitting this term and concentrate on the study of the corresponding interlacing solution
\begin{eqnarray}\label{Linear mSOsln}
x(t)&=&\cos t+\int_0^t \sigma\sin(t-s) dB_s+\int_{|z|< 1} \sigma z\sin(t-s) \widetilde N(ds,dy),\\
y(t)&=&-\sin t+\int_0^t \sigma\cos(t-s) dB_s+\int_{|z|< 1} \sigma z\cos(t-s) \widetilde N(ds,dy).
\end{eqnarray}
\par
By It\^o isometry and the properties of compensated Poisson integral \cite{Ap}, we can find that the second moment of this solution satisfies
\begin{equation}\label{Linear SO Hamiltonian}
\mathbb{E}(x(t)^2+y(t)^2)=1+\sigma^2t+\sigma^2t\int_{|z|< c}|z|^2\nu(dz),
\end{equation}
where $\int_{|z|< c}|z|^2\nu(dz)<\infty$ by the definition of L\'{e}vy motion.
\par
It means that the Hamiltonian here grows linearly with respect to time $t$. This is quite different from the the case of deterministic Hamiltonian systems, for which the Hamiltonian is preserved for all $t$.
\end{example}


\begin{remark}\label{remark 3.3}
 An alternative view of stochastic Hamilton system is that we can regard it as an open Hamiltonian system within the external world: the stochastic part in (\ref{Equation-1}) characterizes the complicated interaction between the ``deterministic" Hamiltonian system with the Hamiltonian $H_0$ and the chaotic environment \cite{Misa}.
 \end{remark}

For stochastic Hamiltonian system with L\'{e}vy noise (\ref{Eqnarray-Q-noise} - \ref{Eqnarray-P-noise}), according to (\ref{Equation-work}), the work done by L\'evy noise is formally
\begin{equation}
W_{stoch}=-\sum_{k=1}^d H_k\diamond \dot{L}^k(t).
\end{equation}
Based on (\ref{Equation-action integral}), we infer the action integral of motion as follows
\begin{equation}\label{Equation-action integral stoch}
\mathcal{S}_{stoch}[\gamma]=\int_{t_0}^{t_1}(L-W_{stoch})dt=\int_{t_0}^{t_1}L(\gamma(t),\dot{\gamma}(t),t)dt-\sum_{k=1}^d\int_{t_0}^{t_1}H_k(\gamma(t),t)\diamond d{L}^k(t),
\end{equation}
where $\gamma=\{ (Q(t),P(t)): t_0\leqslant t \leqslant t_1\}$.
\par
Moreover, by (\ref{Eqnarray-Lagrange1} - \ref{Eqnarray-Lagrange2}), the Euler-Lagrange equations of motion for the stochastic Hamiltonian system with L\'{e}vy noise (\ref{Eqnarray-Q-noise} - \ref{Eqnarray-P-noise}) have the form
\begin{eqnarray}
\frac{d}{dt} \frac{\partial L}{\partial \dot{P}} - \frac{\partial L}{\partial P} &=& \sum_{k=1}^d\frac{\partial H_k}{\partial P}\diamond \dot{L}^k(t),  \label{Eqnarray-Lagrange1_stoch} \\
\frac{d}{dt} \frac{\partial L}{\partial \dot{Q}} - \frac{\partial L}{\partial Q} &=& \sum_{k=1}^d\frac{\partial H_k}{\partial Q}\diamond \dot{L}^k(t).  \label{Eqnarray-Lagrange2_stoch}
\end{eqnarray}
\par
We call $\mathcal{S}_{stoch}$ the stochastic action integral and call (\ref{Eqnarray-Lagrange1_stoch} - \ref{Eqnarray-Lagrange2_stoch}) the stochastic Euler-Lagrange equations.

\begin{theorem}\label{theorem 3.2}  (Hamilton's Principle)
The paths that are realized by the stochastic dynamical system represented by stochastic Euler-Lagrange equations (\ref{Eqnarray-Lagrange1_stoch} - \ref{Eqnarray-Lagrange2_stoch}) are those for which the stochastic action integral (\ref{Equation-action integral stoch}) is stationary for fixed endpoints $\gamma(t_0)=(Q_0,P_0)$ and $\gamma(t_1)=(Q_1,P_1)$.

\begin{proof}
The action $\mathcal{S}_{stoch}[\gamma]$ is stationary if it does not vary when the curve is slightly changed, $\gamma(t)\to\gamma(t)+\delta\gamma(t)$. The change in the action upon doing this can be formally expanded in $\delta\gamma$,
\begin{equation}
\mathcal{S}_{stoch}[\gamma+\delta\gamma]-\mathcal{S}_{stoch}[\gamma]=\int_{t_0}^{t_1}\frac{\delta\mathcal{S}_{stoch}}{\delta\gamma}\delta\gamma(t)dt+o(\delta\gamma),
\end{equation}
where $\delta\mathcal{S}_{stoch}/\delta\gamma$ is called the Fr\'echet or functional derivative of $\mathcal{S}_{stoch}$.
\par
Applying the chain rule for the Marcus integral, we calculate the derivative,
\begin{align}
\delta \mathcal{S}_{stoch}
=& \int_{t_0}^{t_1}(\frac{\partial L}{\partial Q}\delta Q+\frac{\partial L}{\partial P}\delta P+\frac{\partial L}{\partial \dot{Q}}\delta \dot{Q}+\frac{\partial L}{\partial \dot{P}}\delta \dot{P})dt \notag\\
&-\sum_{k=1}^d  \int_{t_0}^{t_1}(\frac{\partial H_k}{\partial Q}\delta Q+\frac{\partial H_k}{\partial P}\delta P)\diamond \dot{L}^k(t)dt\notag\\
=& \left[ \frac{\partial L}{\partial \dot{Q}} \delta Q\right]_{t_0}^{t_1} + \left[ \frac{\partial L}{\partial \dot{P}} \delta P\right]_{t_0}^{t_1} +
 \int_{t_0}^{t_1}(\frac{\partial L}{\partial Q}-\frac{d}{dt}\frac{\partial L}{\partial \dot{Q}}-\sum_{k=1}^d  \frac{\partial H_k}{\partial Q}\diamond \dot{L}^k(t))\delta Q dt\notag\\
& + \int_{t_0}^{t_1}(\frac{\partial L}{\partial P}-\frac{d}{dt}\frac{\partial L}{\partial \dot{P}}-\sum_{k=1}^d  \frac{\partial H_k}{\partial P}\diamond \dot{L}^k(t))\delta P dt.\notag
\end{align}
The boundary terms vanish because the endpoints of $\gamma(t)$ are fixed: $\delta Q(t_0)=\delta Q(t_1)=\delta P(t_0)=\delta P(t_1)=0$. As discussed in Wang et al \cite{Wang}, the desired result follows.
\end{proof}
\end{theorem}


\begin{example}\label{example 3.2}  
Consider the linear stochastic oscillators with L\'{e}vy noise (\ref{Linear SO1} - \ref{Linear SO2}). We show that the equations (\ref{Linear SO1} - \ref{Linear SO2}) are equivalent to the stochastic Euler-Lagrange equations of motion with L\'{e}vy noise (\ref{Eqnarray-Lagrange1_stoch} - \ref{Eqnarray-Lagrange2_stoch}).
Indeed, by the relation between Lagrangian and Hamiltonian, we have
$$
L(x,y,\dot{x},\dot{y})=x\cdot \dot{y}-H(x,y)=x\cdot \dot{y}-\frac{1}{2}(y^2+x^2).
$$
According to (\ref{Eqnarray-Lagrange1_stoch} - \ref{Eqnarray-Lagrange2_stoch}), the Euler-Lagrange equations of motion of the linear stochastic oscillators have the form
\begin{equation}\label{Equation-SOL}
\left\{
\begin{array}{rl}
&\frac{d}{dt} \frac{\partial L}{\partial {y}} - \frac{\partial L}{\partial x} = -\sigma\dot{L}_t,  \\
&\frac{d}{dt} \frac{\partial L}{\partial \dot{y}} - \frac{\partial L}{\partial y} =  0.
\end{array}
\right.
\end{equation}
since $H_1=\sigma x$. With initial conditions $x(0)=x_0$, $y(0)=y_0$, (\ref{Equation-SOL}) are equivalent to the Hamiltonian equations of motion (\ref{Linear SO1} - \ref{Linear SO2}).
\end{example}

Consider the stochastic action integral $\mathcal{S} $ in (\ref{Equation-action integral stoch}) as a function of the two endpoints $(Q(t_0),\dot{Q}(t_0))=(Q_0,\dot{Q}_0)$ and $(Q(t_1),\dot{Q}(t_1))=(Q_1,\dot{Q}_1)$. We have the following theorem which plays an important role in constructing some numerical methods \cite{Mi2, Mi,  Wang, Fox}.

\begin{theorem}\label{theorem 3.3} (Characterization of stochastic action integral) 
The stochastic action integral $\mathcal{S}_{stoch} $ satisfies
\begin{equation}\label{Equation-Sp0p1}
d \mathcal{S}_{stoch}=-P_0^TdQ_0+P_1^TdQ_1.
\end{equation}
Furthermore, if the Lagrangian $L$ and the functions $H_k$ $(k = 1,...,d)$ are sufficiently smooth with respect to $P$ and $Q$,
then the mapping $$(P_0,Q_0) \mapsto (P_1,Q_1)$$ defined by equation (\ref{Equation-Sp0p1}) is symplectic.
\end{theorem}
The proof is given in the Appendix.



\renewcommand{\theequation}{\thesection.\arabic{equation}}
\setcounter{equation}{0}

\section{An averaging principle for integrable stochastic Hamiltonian systems}
\noindent
We now return to  the stochastic Hamiltonian systems with L\'{e}vy noise (\ref{Equation-1}) on a $2n$-dimensional smooth manifold $M$ (for simplicity, set $n=d$ in the rest of this discussion). 
As mentioned earlier, such systems are themselves nonconservative systems with the perturbation of L\'{e}vy noise.
Then a interesting question to raise is:  if there is even a  small external perturbation in this stochastic system, just as the deterministic Hamiltonian case and the stochastic Hamiltonian case with Brownian noise refering to the study of Freidlin-Wentzell \cite{Fr}, Li \cite{Li} and so on, what the effective dymanic behaviour would be? To answer this question, we consider the (completely) integrable stochastic Hamiltonian systems with L\'{e}vy noise.\\

Recall that on a $2d$-dimensional smooth manifold, a family of $d$ smooth Hamiltonians $\{H_k\}_{k=1}^{d}$ is said to form a (completely) integrable system if they are pointwise Poisson commuting and if the corresponding Hamiltonian vector fields $V_{k}$ are linearly independent at almost all points. 
\par
We call systems (\ref{Equation-1}) \emph{ (completely) integrable stochastic Hamiltonian systems with L\'{e}vy noise}, if they satisfy the following condition:
 \emph{
\begin{itemize}
\item[\bf{ A3}] Completely integrability: $\{H_k\}_{k=1}^{d}$ is an integrable family, and Hamiltonian vector field $V_0$ with Hamiltonian $H_0$ is commuting with the family of vector fields $V_{k}$. That is, $dH_j(V_{i})=\omega^2(V_{i}, V_{j}) = 0$ for $i,j=0,1,2,...,d$.
\end{itemize}
}
\par
For the sake of convenience and readability, in the sense of of L\'evy-It\^o decomposition and Marcus integral (\ref{Marcus}), we consider the following integrable stochastic Hamiltonian system with L\'{e}vy noise, which satisfies assumptions {\bf A1} - {\bf A3},
\begin{equation}\label{Equation-H}
dX_t=V_0(X_t)dt+\sum_{k=1}^d V_{k}(X_t)\circ dB^k(t)+\sum_{k=1}^d V_{k}(X_t)\diamond dL^{k}(t),~ X(t_0)=x\in M.
\end{equation}
Where $B(t)$ is a $d$-dimensional independent standard Brownian motion, $L(t)$ is a $d$-dimensional independent L\'{e}vy motion with the generating triplet $(0, 0, \nu)$ which is a pure jump process.

\subsection{Invariant manifolds and invariant measure for integrable stochastic Hamiltonian systems}
\noindent
Due to the system has $d$ first integrals $H_1, . . . ,H_d$ in involution. We consider the joint integral level
\begin{equation}\label{Equation-levelset}
M_h =\{x \in M: H_i(x) = h_i= const, ~i=1,2,..d\}.
\end{equation}
The Liouville-Arnold theorem \cite{Ar} indicates that if the functions $H_i$ on $M_h$ are independent, then each compact connected component of $M_h$ is diffeomorphic to a $d$-dimensional torus $\mathbb{T}^d$.
It remains to use the geometric fact: in this integrable system there are convenient, so-called, action-angle coordinates $(I,\theta)$ ($I$ are the actions and $\theta$ are the angles) such that
$\omega^2 = dI \wedge d\theta$ (symplecticity), $H = H(I)$ (i.e., $I$ are first integrals).
\par
We next show that a solution to these SDEs preserves the energies $H_i$ and there are corresponding invariant manifolds (level sets). Let $\Psi_t:=(\Psi(t,\omega,x),t\geqslant 0)$ be the solution flow of the SDE (\ref{Equation-H}) with starting point $x$ and $(T_t, t \geqslant 0)$ be the semigroup associated with $\Psi_t$. Applying the chain rule for the Stratonovith itegral and Marcus integral, and using the assumption {\bf A3} of completely integrability, we have
\begin{lemma}\label{lemma 4.1}
The solution flow $\Psi_t:=(\Psi(t),t\geqslant 0)$ of SDE (\ref{Equation-H}) preserves the invariant manifolds $M_h$, i.e. for $1\leqslant i \leqslant d$,
\begin{equation}
dH_i(X_t)=dH_i(V_0(X_t))dt+\sum_{k=1}^d dH_i(V_{k}(X_t)\circ dB^k(t)+\sum_{k=1}^d dH_i(V_{k}(X_t))\diamond dL^{k}(t)=0. \notag
\end{equation}
\end{lemma}
Indeed, for each $x$ in $M$, we have $h = (H_1(x), . . . , H_n(x))$, thus it determines an invariant manifold, which we write also as $M_{H (x)}$. Note that the $d$ vector fields $\{V_k\}_{k=1}^{d}$ are tangent to $M_{H (x)}$ and the symplectic form $\omega^2$ vanishes on the invariant manifolds $M_h$. The Markovian solution to SDEs (\ref{Equation-H}) restricts to each invariant manifold and the generator $\mathcal{A}$ of restriction is the sum of a second-order elliptic differential operator and a (compensated) integral of difference operator, i.e.,
\begin{align}\label{Eqnarray-generator}
 (\mathcal{A}f)(x)=&(\mathcal{L}_0f)(x)+\frac{1}{2}\sum_{k=1}^d(\mathcal{L}_k\mathcal{L}_k f)(x) \notag\\
 &+ \int_{{\mathbb{R}^d} \setminus \{ 0 \}} [f(\phi(z)x)-f(x)-\sum_{k=1}^d z^k(\mathcal{L}_k f)(x)\mathbf{ 1}_{\{|z|<1\}}(z)]\nu(dz)
\end{align}
for every function $f\in C_b^2(M)$. Here we denote as $\mathcal{L}_0$, $\mathcal{L}_k$ the Lie differentiation in the direction of $V_0$,$V_k$, respectively, and $C_b^2(M)$ the collection of all bounded Borel measurable $C^2$ functions on $M$. More precisely, we have $\mathcal{L}f=df(V_0 ) = \omega^2(v_{f},V_0)$ and $\mathcal{L}_kf=df(V_k ) = \omega^2(v_{f},V_k)$.\\

We remark that an invariant probability measure for (\ref{Equation-H}) is by definition a Borel probability measure on $M$ such that
\begin{equation}
\int_{M}(T_t g)(x)\mu(dx)=\int_{M}g(x)\mu(dx) \notag
\end{equation}
for all $t>0$, $g\in C^1(M)$. Based on the celebrated Krylov-Bogoliubov method, we have the following lemma.
\begin{lemma}\label{lemma 4.2} (\cite[Theorem 4.5]{Br})
If $M$ is locally compact in the relative topology and assumptions {\bf A1} and {\bf A2} hold, then the system (\ref{Equation-H}) has at least one invariant measure.
\end{lemma}

For simplicity, throughout this paper, we assume that:
\emph{
\begin{itemize}
\item[\bf{A4}] The invariant manifolds are compact, the map $H: x \in M \to (H_1(x), . . . , H_d(x)) \in \mathbb{R}^d$ is proper, and its set of critical points has measure zero.
\end{itemize}
}
Under our assumption, for almost every point $h_0$ in $\mathbb{R}^d$, there is a neighbourhood $N$ of $h_0$ such that $H^{-1}(h)$ is a smooth sub-manifold for all $h\in N$ and that there is a diffeomorphism from
$H^{-1}(N)$ to $N\times H^{-1}(h_0)$. We call such $h_0$ a regular value of $H$, and call the point $y$ in $M$ a critical point if $H(y)$ is not regular. By Morse-Sard theorem \cite{Ar2}, the set of critical values of the function $H$ has measure zero.
\par
Recall that in a neighbourhood of a regular point $h_0$ of $H$, every component of the level set $M_{h_0}$ is diffeomorphic to a $d$-dimensional torus $\mathbb{T}^d$, and a small neighbourhood $U_0$ of $M_{h_0}$ is diffeomorphic to the product space $\mathbb{T}^d\times D$, where $D$ is a relatively compact open set in $\mathbb{R}^d$. Take an action-angle chart around $M_{h}$. The measure $(\sum_i dI^i\wedge d\theta^i)^d$ on the product space naturally splits to give us a probability measure,
the Haar measure \cite{Ar2} $\theta_1 \wedge ... \wedge \theta^d$ on $\mathbb{T}^d$. We take the corresponding one on $M_{h}$ and denote it by $\mu_h$, just like the case of Brownian in \cite{Li}. With the help of action-angle transformation and the above assumptions, we thus have the following lemma.

\begin{lemma}\label{lemma 4.2}
Assume that assumptions {\bf A1} - {\bf A4} are in force. Let $E=span\{ V_1,...,V_d\}$ be a sub-bundle of the tangent bundle of rank $d$. Let $U$ be a section of $E$ commuting with all $V_i$ $(1\leqslant i \leqslant d)$. The invariant measure for stochastic Hamiltonian system (\ref{Equation-H}) restricted to the invariant manifold $M_{h}$ is $\mu_h$, which varies smoothly with $h$ in sufficiently small neighbourhoods of a regular value.
\end{lemma}
\begin{proof}
Recall that $M_h$ have the form in (\ref{Equation-levelset}), we rewrite $U=\sum_{i=1}^{d}h_iV_i(x)$. For any smooth function $f$ on $M_h$, we have
\begin{align}
\int_{M_h}df(V_i)(x)\mu_h(dx)
&=\int_{\mathbb{T}^d}d(f \circ \varphi)\Big( -\sum_{k=1}^d \frac{\partial (H_k\circ \varphi)}{\partial I_k} \frac{\partial}{\partial\theta_k}\Big)d\theta \notag\\
&=-\sum_{k=1}^d \omega_k^i(I)\int_{\mathbb{T}^d}\Big( \frac{\partial}{\partial\theta_k}(f \circ \varphi)\Big)d\theta =0,\notag
\end{align}
where $\varphi^{-1}$ is the action-angle coordinate map (see the next subsection for detail), $(I,\theta)$ are the corresponding action-angle coordinates. Thus $U$ is divergence free, i.e. ${\rm div}_E U=0$, in the sense of
\begin{align}
\int_{M_h}df(U)(x)\mu_h(dx)=-\int_{M_h} {\rm div}_E U(x)\mu_h(dx)=0.
\end{align}
Therefore, restricted to the torus, the invariant measure of SDE (\ref{Equation-H}) is the same as that of the corresponding SDE without a drift. From the action-angle transformation we find that the measure $\mu_h$ is the desired object. 
\end{proof}

\subsection{The perturbed system and statement of an averaging priciple}
\noindent
We next study the situation where an integrable stochastic Hamiltonian system is perturbed by a transversal smooth vector field and the stochastic differentials. Let $y_0$ be a regular point of $H$ in $M$ with a neighborhood $U_0$ the domain of an action-angle coordinate map:
$$
\varphi^{-1}:U_0\to  \mathbb{T}^d\times D
$$
where $\mathbb{T}^d$ is an $d$-dimensional torus and $D$ is a relatively compact  open set of $\mathbb{R}^n$. Note that the action coordinate of a point $x\in U_0$ can be denoted with the help of the projection  $\pi:U_0 \to D$ by $\varphi^{-1}(x)=(\theta^{\ast},\pi(x))$ for some $\theta^{\ast}\in\mathbb{T}^d$. 
We consider the perturbed system corresponding to (\ref{Equation-H}):
\begin{align}\label{Eqnarray-perturbation}
dY_t^{\varepsilon} =& V_0(Y_t^{\varepsilon})dt+\sum_{k=1}^d V_{k}(Y_t^{\varepsilon})\circ d{B_t^k}+ \sum_{k=1}^d V_{k}(Y_t^{\varepsilon})\diamond dL^{k}(t) \notag\\
&+ \varepsilon \Big(K(Y_t^{\varepsilon})dt+\sum_{k=1}^d  F_k(\pi(Y_t^{\varepsilon}))\circ d{\tilde{B}_t^k}+ \sum_{k=1}^d {G}_k(\pi(Y_t^{\varepsilon}))\diamond d\tilde{L}_t^k\Big)
\end{align}
with initial condition $Y_0^{\varepsilon} = y_0$. 
Where $K$ is a smooth and global Lipschitz continuous vector field, transversal in sense that $\omega^2(V_{k}, K)$, $k=0, 1,...,d$, are not all identically zero; $\tilde{B}(t)$ is a $d$-dimensional independent standard Brownian motion; $\tilde{L}(t)$ is a $d$-dimensional independent pure jump L\'{e}vy motion with the generating triplet $(0, 0, \nu^{\prime})$. Moreover, $F$, $G$ are smooth vector fields such that $F$, $\hat{D}{F}$, $G$ and $\hat{D}{G}$ are globally Lipschitz continuous.
\par
We denote by $Y_t^\varepsilon$ the solution to (\ref{Eqnarray-perturbation}) and by $X_t=Y_t^0$ the solution to  (\ref{Equation-H}) with initial value $y_0$. In the action-angle coordinate, $X_t=\varphi(\theta_t,I_t)$, $\theta\in  \mathbb{T}^d$, $I\in D$ and $Y_t^\varepsilon=\varphi(\theta_t^\varepsilon,I_t^\varepsilon)$, $\theta^\varepsilon\in  \mathbb{T}^d$, $I^\varepsilon\in D$. Let $\tilde{H}_k=H_k(\varphi(\theta_t,I_t))$ be the induced Hamiltonian on $\mathbb{T}^d\times D$,
then, for $i=1,...,d$,
\begin{eqnarray}
\dot{\theta}_k^i&=&\frac{\partial \tilde{H}_k}{\partial I_i}=:\omega_k^i(I) ,\notag\\
\dot{I}_k^i&=&-\frac{\partial \tilde{H}_k}{\partial \theta_i}=0,\notag
\end{eqnarray}
with $\omega_k^i$ smooth functions. Indeed, the corresponding induced Hamiltonian vector field $\tilde{V}_k:=V_{\tilde{H}_k}=-\sum_{i=1}^d({\partial (H_k\circ \varphi)}/{\partial I_i})({\partial}/{\partial \theta_i})$.
\par
For the perturbed SDE (\ref{Eqnarray-perturbation}), we write the induced perturbation vector field of $K$ as $(K_\theta, K_I)$ on $\mathbb{T}^d \times D$ with $K_{\theta}=(K_{\theta}^{1}, ... , K_{\theta}^{d})$ and $K_I=(K_{I}^{1}, ... , K_{I}^{d})$ the angle and action component, respectively, and we do the same thing for $F$ and ${G}$. By the chain rule for Stratonovitch integral as well as that for Marcus integral, we have the following form of the SDE on $\mathbb{T}^d\times D$:
\begin{align}
d\theta_{t}^{\varepsilon}=&\omega_0(I_t^\varepsilon)dt+\sum_{k=1}^{d}\omega_k(I_t^\varepsilon)\circ d{B_t^k}+\sum_{k=1}^{d}\omega_k(I_t^\varepsilon)\diamond dL^k(t)\notag\\
&+\varepsilon \Big(K_\theta( \theta_t^\varepsilon,I_t^\varepsilon)dt
+ \sum_{k=1}^d{F}_{\theta,k}(I_t^\varepsilon)\circ d{\tilde{B}_t^k}
+ \sum_{k=1}^d{G}_{\theta,k}(I_t^\varepsilon)\diamond d{\tilde{L}_t^k}\Big) \label{Eqnarray-I-theta1},\\
dI_{t}^{\varepsilon}=&\varepsilon \Big(K_I( \theta_t^\varepsilon, I_t^\varepsilon)dt
+ \sum_{k=1}^d{F}_{I,k}(I_t^\varepsilon)\circ d{\tilde{B}_t^k}
+ \sum_{k=1}^d{G}_{I,k}(I_t^\varepsilon)\diamond d{\tilde{L}_t^k}\Big). \label{Eqnarray-I-theta2}
\end{align}
\par
Note that subjected to a small perturbation, the system splits into two parts with fast rotation along the nonperturbed trajectories and slow motion across them, so it's a situation where the averaging principle is to be expected to hold. 
\par
For this purpose, we further adopt the following assumptions:
\emph{
\begin{itemize}
\item[\bf{A5}] There is a constant $p\geqslant2$ such that the L\'evy measures $\nu$ (of $L_t$) and $\tilde{\nu}$ (of $\tilde{L}_t$) satisfy
$$
\int_{\mathbb{R}^d}|z|^p\nu(dz)<\infty, ~~~~ and ~~~~\int_{\mathbb{R}^d}|z|^{2p}\tilde{\nu}(dz)<\infty.
$$
\item[\bf{A6}]  For any continuous function $f$ on the compact manifold converging to infinity when $t$ converges to infinity, $\frac{1}{t}\int_s^{s+t}f(X_r)dr\to\int_{M_h} f(z)\mu_h(z)$ when $t\to \infty$, in $L^p$ $(p \geqslant 2)$, and the rate of convergence, denoted by $\eta(t)$, is a positive, bounded, decreasing function from $[0,\infty)$ to $[0,\infty)$ with $\eta(t)\searrow 0$ as $t \to \infty$.
\end{itemize}
}
Some comments on these two assumptions have to be made: Note that the invariant manifold here is actually $d$-dimensional torus, which is compact and bounded. It is necessary and reasonable to put forward assumption {\bf A5} referring to \cite{Hog2}. This assumption indicates the polynomial moments of $L(t)$ and $\tilde{L}(t)$ exist, and will play an important role in estimating some terms of the Marcus equation in the next subsection. Note that the motion on the torus, which would be quai-periodic if there are no diffusion terms, is ergodic. Indeed, there is no standard rate of convergence for general Markovian systems in the ergodic theorem; see e.g. Krengel \cite{KreU}, Kakutani and Petersen \cite{KSPK}. It is natural to deal with an averaging principle in the terms of the function $\eta$ following the approach in Freidlin-Wentzell \cite{Fr}. We thus have the ergodicity assumption {\bf A6}. More information on rates of convergence for L\'evy noise driven systems can be found in Kulik \cite{Kulik} and H\"ogele-de Costa \cite{Hog2}, and a detailed example will be shown in subsection 4.5. \\
\par
To study slow motion governed by the transversal part of the vector field $K$ and the stochastic differentials $F\circ \dot{\tilde{B}}_t$, $G \diamond \dot{\tilde{L}}_t$, it is convenient to rescale the time, see Lemma \ref{lemma 4.4} for detail. Denote $Y_{t/\varepsilon}^{\varepsilon}$ the process scaled in time by $1/\varepsilon$ which coincides, in the sense of probability distributions \cite{Fr}, with $Y_{t}^{\varepsilon}$.
Then, the evolution of $Y_{t/\varepsilon}^{\varepsilon}$ is the skew product of the fast diffusion of order $\frac{1}{\varepsilon}$ along the invariant manifold  and the slow diffusion of order 1 across the invariant manifold. We finally obtain a new dynamical system in the limit as $\varepsilon$ goes to zero: Compared with the motion in the transversal direction, the motion along the torus is significantly faster, thus as the randomness in the fast component is averaged out by the induced invariant measure, the evolution of the action component of $Y_{t/\varepsilon}^{\varepsilon}$ will have a limit. 
\par
The main theorem on averaging principle for (completely) integrable stochastic Hamiltonian system is formulated below, and the detail proof is shown in next subsection.

\begin{theorem}\label{theorem 4.1} (Averaging Principle) 
Consider the perturbed SDE (\ref{Eqnarray-perturbation}) with initial value $Y_0^{\varepsilon} = y_0$ and satisfying assumptions {\bf A1} - {\bf A6} for some $p\geqslant 2$. Set $H_i^\varepsilon(t)=H_i(Y_{t/\varepsilon}^{\varepsilon})$, for $i=1,2,...,d$. Define exit time $\tau^\varepsilon:=\inf \{ t\geqslant 0: Y_{t/\varepsilon}^{\varepsilon}\notin U_0\} $ as the first time that the solution $Y_{t/\varepsilon}^{\varepsilon}$ exists from $U_0$.
\par
Let $\bar{H}(t)=(\bar{H}_1(t), ... ,\bar{H}_d(t))$ be the solution to the following system of $d$ deterministic differential equations
\begin{align}\label{Eqnarray-aver}
d\bar{H}_i(t)=\int_{M_{\bar{H}(t)}}\omega^2(V_{i}, K)(\bar{H}(t),z) \mu_{\bar{H}_t}(dz)dt+\sum_{k=1}^d F_k(\bar{H}(t))\circ d\tilde{B}_t^k+\sum_{k=1}^d G_k(\bar{H}(t))\diamond d\tilde{L}^{k}(t) 
\end{align}
with initial value $\bar{H}(0)=H(y_0)$. Define exit time $\tau^0:=\inf \{ t\geqslant 0 : \bar{H}(t) \notin U_0\} $ as the first time that $\bar{H}(t)$ exists from $U_0$.
\par
Then we have that:
\begin{itemize}
\item[(1)] For any sufficiently small $\varepsilon>0$ and $t<\tau_0$, there exist constants $k_1$, $k_2$, $k_3>0$ such that
\begin{eqnarray}
\bigg( \mathbb{E}\big[ \sup_{s\leqslant t} |H^\varepsilon(s\wedge \tau^\varepsilon)-\bar{H}(s\wedge \tau^\varepsilon)|^p\big] \bigg)^\frac{1}{p}\leqslant k_1 t \big(\varepsilon^{1-k_2 t} +\eta(t|\ln\varepsilon|)\big)\exp(k_3 t).
\end{eqnarray}
\item[(2)] If there exists a $r>0$ such that $U_{r}:=\{ x\in M : |H(x)-H(y_0)|\leqslant r\}\subset U_0$. Define exit time $\tau_\delta:=\inf \{ t\geqslant 0 : |\bar{H}_t-H(Y_0)|\geqslant r-\delta\}$ for $\delta>0$.
Then for any $\delta>0$, constant $k_2>0$ given above, and constants $k_4$, $k_5$ depending on $\tau_\delta$,
\begin{eqnarray}
\mathbb{P}(\tau^\varepsilon<\tau_\delta)\leqslant k_4\delta^{-p} {\tau_\delta}^p \big(\varepsilon^{1-k_2 \tau_\delta} +\eta(\tau_\delta|\ln\varepsilon|)\big)^p\exp(k_5 \tau_\delta).
\end{eqnarray}
\end{itemize}
\end{theorem}

\begin{remark}\label{remark 4.6}
This result includes the case of pure Gaussian noise and case of pure jump noise, where the former situation has been considered, cf. Li \cite[Theorem 3.3.]{Li}. Indeed, Hamiltonian vector $V_0$ in (\ref{Equation-H}) can be weakened to be a locally Hamiltonian vector which is not given by a Hamiltonian function as in \cite{Li}. The main different between Gaussian situation and the situation we considered here comes from the estimation for L\'evy noise term. However, if deterministic part of the perturbation is a (local) Hamiltonian vector field with $\omega^2(V_{i}, K)=0$, or the multiplicative coefficients of the stochastic differentials are not only depend on the slow component, the situation will become more complex. To deal with these problem on multiplicative L\'evy noise is still remain to solve.
\end{remark}

\subsection{Proof of the averaging priciple}
\noindent
In this subsection we always assume that assumptions {\bf A1} - {\bf A6} are in force for some $p\geqslant 2$. We first get the information on the order of which the first integrals for the perturbed system change over a time interval by next lemma.

\begin{lemma}\label{lemma 4.4} 
Let $\tau^\varepsilon=\inf \{ t\geqslant0:Y_t^\varepsilon \notin U_0\} $. For any Lipschitz test function $f : M\to\mathbb{R}$ and $p\geqslant2$, we have
\begin{equation}\label{Equation-41}
\left[
\mathbb{E}(\sup_{s\leqslant t\wedge \tau^\varepsilon} | f(Y_s^\varepsilon)-f(X_s) |^p)
\right]^{\frac{1}{p}}\leqslant C_1\varepsilon e^{C_2 t},
\end{equation}
where $C_1$, $C_2$ are constants depending on the Lipschitz coefficient of $f$, on the upper bounds of the norms of vector fields $K$, $F$, $G$, $V_k$, $k=0,...,d$ and their derivatives with respect to the action-angle coordinate on $ T^d\times D$.
\end{lemma}
\begin{proof}
In action-angle coordinates, we rewrite the flows as  $X_t = \varphi(\theta_t, I_t)$ and $Y_t^\varepsilon = \varphi(\theta_t^\varepsilon, I_t^\varepsilon)$. And the corresponding SDEs on $ T^d\times D$ under the action-angle coordinate map are shown in (\ref{Eqnarray-I-theta1})-(\ref{Eqnarray-I-theta2}). Since $D$ is relatively compact, $\partial({f}\circ\varphi)/\partial \theta$ and $\partial({f}\circ\varphi)/\partial I$ are bounded on $ T^d\times D$. We thus obtain
\begin{eqnarray}
|f(Y_t^\varepsilon)-f(X_t)|&=&|f\circ\varphi( \theta_t^\varepsilon, I_t^\varepsilon)-f\circ\varphi(\theta_t, I_t)| \notag\\
&\leqslant& c_0|(\theta_t^\varepsilon- \theta_t, I_t^\varepsilon-I_t)|
\leqslant c_0| \theta_t^\varepsilon- \theta_t|+c_0|I_t^\varepsilon-I_t|,
\end{eqnarray}
for some constant $c_0>0$.
\\
{\bf Estimate of the action component $|I_t^\varepsilon-I_t|$.}  ~
Note that the facts that equation (\ref{Eqnarray-I-theta1}) satisfies the chain rule in the sense of Stratonovitch and Marcus, and $  \langle Dh(x),u  \rangle=p|x|^{p-2} \langle x,u  \rangle$ for the function $h(x)=|x|^p$, we obtain, for $s<\tau^\varepsilon$,
\begin{align}
| I_s^\varepsilon-I_s|^p=
&\varepsilon p \Big( \int_0^s | I_r^\varepsilon-I_r|^{p-2} \left \langle I_r^\varepsilon-I_r,  K_I( \theta_t^\varepsilon, I_r^\varepsilon)\right \rangle dr\notag\\
& +\int_0^s | I_r^\varepsilon-I_r|^{p-2}   \sum_{k=1}^d\left \langle I_r^\varepsilon-I_r,   F_{I,k}(I_r^\varepsilon) \circ d{B_r^k}\right \rangle \notag\\
& +\int_0^s | I_{r-}^\varepsilon-I_{r-}|^{p-2}  \sum_{k=1}^d\left \langle I_{r-}^\varepsilon-I_{r-},   {G}_{I,k}(I_{r-}^\varepsilon) \diamond d{\tilde{L}_r^k}\right \rangle \Big)\notag\\
\leqslant
&\varepsilon p  \int_0^s | I_r^\varepsilon-I_r|^{p-2} \Big| \left \langle I_r^\varepsilon-I_r,  K_I(\theta_t^\varepsilon,I_r^\varepsilon)+\frac{d}{2}\hat{D}_IF_I(I_r^\varepsilon)\right \rangle \Big|dr\tag{$\Sigma_1$}\\
& + \varepsilon p\sum_{k=1}^d\int_0^s | I_r^\varepsilon-I_r|^{p-2} \big| \left \langle I_r^\varepsilon-I_r,   F_{I,k}(I_r^\varepsilon) d{B_r^k}\right \rangle\big|\tag{$\Sigma_2$}\\
& + \varepsilon p\sum_{k=1}^d\int_0^s | I_{r-}^\varepsilon-I_{r-}|^{p-2} \big| \left \langle I_{r-}^\varepsilon-I_{r-},   {G}_{I,k}(I_{r-}^\varepsilon) d{\tilde{L}_r^k} \right \rangle\big|\Big)\tag{$\Sigma_3$}\\
& + \varepsilon p \sum_{k=1}^d\sum_{0\leqslant r\leqslant s} \int_0^s | I_{r-}^\varepsilon-I_{r-}|^{p-1} dt
\big| \phi(\Delta L^k (r), {G}_{I,k}(I_{r-}^\varepsilon), I_{r-}^\varepsilon) \notag\\
&- I_{r-}^\varepsilon-{G}_{I,k}(I_{r-}^\varepsilon)\Delta L^k (r)\big|\tag{$\Sigma_4$}\\
=&\Sigma_1+\Sigma_2+\Sigma_3+\Sigma_4, \label{Align-Sigma}
\end{align}
where $(\hat{D}_I{F}_{I})_i=\max_{1\leqslant k \leqslant d} |{F}_{I,k}\cdot \nabla_I {F}_{I,k}^{i}|$ comes from the Stratonovitch correction. 
By assumption the induced vector fields and their derivatives are bounded on $ T^d\times D$. A direct computation gives
\begin{align}\label{Align-Sigma_1}
\Sigma_1\leqslant \varepsilon p \big(\sup_{ T^d\times D} |K_I|+\frac{d}{2}\sup_{ T^d\times D}|\hat{D}_I{F}_{I}|\big)\int_0^s | I_r^\varepsilon-I_r|^{p-1} dr.
\end{align}
Note that the term $\Sigma_3$ has the representation with respect to the compensated Possion random measure $\tilde{N}^{\prime}$ associated to $\tilde{L}(t)$ \cite{Ap, Ku}, we have
\begin{align}\label{Align-Sigma_3}
\Sigma_3
=&\varepsilon p \sum_{k=1}^d\int_0^s \int_{\mathbb{R}^d\setminus \{ 0 \}}| I_{r-}^\varepsilon-I_{r-}|^{p-2} \big| \left \langle I_{r-}^\varepsilon-I_{r-},   {G}_{I,k}(I_{r-}^\varepsilon)z\right \rangle {\tilde{N}(dr,dz)} \big|\notag\\
&+\varepsilon p \sum_{k=1}^d\int_0^s \int_{|z|> 1}| I_{r-}^\varepsilon-I_{r-}|^{p-2} \big| \left \langle I_{r-}^\varepsilon-I_{r-},   {G}_{I,k}(I_{r-}^\varepsilon)z\right \rangle \nu^{\prime}(dz)dr \big|\notag\\
\leqslant&\varepsilon p \sum_{k=1}^d\int_0^s \int_{\mathbb{R}^d\setminus \{ 0 \}}| I_{r-}^\varepsilon-I_{r-}|^{p-2} \big| \left \langle I_{r-}^\varepsilon-I_{r-},   {G}_{I,k}(I_{r-}^\varepsilon)z\right \rangle\big| {\tilde{N}(dr,dz)} \notag\\
&+\varepsilon p d \sup_{ T^d\times D} |{G}_I| \int_{|z|> 1}| z|\nu^{\prime}(dz)\int_0^s| I_{r-}^\varepsilon-I_{r-}|^{p-1}dr.
\end{align}
As smooth vector fields ${G}_I$ and $\hat{D}{G}_I$ are globally Lipschitz continuous. For the last term, by exploiting that $\int_{|z|> 1}| z|^4\nu^{\prime}(dz)<\infty$, we have the following estimation referring to \cite[Lemma 3.1]{Hog2}:
\begin{align}\label{Align-Sigma_4}
\Sigma_4
\leqslant&\varepsilon^2 c_1(p,{G}_I,\hat{D}{G}_I) \sum_{k=1}^d\sum_{0\leqslant r\leqslant s}| I_{r-}^\varepsilon-I_{r-}|^{p-1}|\triangle\tilde{L}^k(r)|^4 \notag\\
\leqslant&\varepsilon^2 c_1\Big(\int_0^{  s}  \int_{\mathbb{R}^d\setminus \{ 0 \}}|I_{r-}^\varepsilon-I_{r-}|^{p-1}|z|^4\tilde{N}^{\prime}(dr,dz)+ \int_0^{  s} | I_{r-}^\varepsilon-I_{r-}|^{p-1}dr\Big).
\end{align}
Combining the estimates (\ref{Align-Sigma_1}), (\ref{Align-Sigma_3}) and (\ref{Align-Sigma_4}), we can find that
\begin{align}\label{Align-Sigma0}
| I_s^\varepsilon-I_s|^p\leqslant
&c_2\varepsilon\int_0^{  s} | I_{r-}^\varepsilon-I_{r-}|^{p-1}dr 
 + \varepsilon p\sum_{k=1}^d\int_0^s | I_r^\varepsilon-I_r|^{p-2} \big| \left \langle I_r^\varepsilon-I_r,   F_{I,k}(I_r^\varepsilon) \right \rangle\big|d{B_r^k}\notag\\
&+\varepsilon p \sum_{k=1}^d\int_0^s \int_{\mathbb{R}^d\setminus \{ 0 \}}| I_{r-}^\varepsilon-I_{r-}|^{p-2} \big| \left \langle I_{r-}^\varepsilon-I_{r-},   {G}_{I,k}(I_{r-}^\varepsilon)z\right \rangle\big| {\tilde{N}(dr,dz)} \notag\\
&+\varepsilon^2 c_3\int_0^{  s}  \int_{\mathbb{R}^d\setminus \{ 0 \}}|I_{r-}^\varepsilon-I_{r-}|^{p-1}|z|^4\tilde{N}^{\prime}(dr,dz).
\end{align}
In order to calculate estimate of the expectation of the supremum for the equation above, it is natural to use It\^o isometry for the Brownian term and use Kunita's first inequality (\cite[Page 265]{Ap}) or other maximal inequality for the compensated Possion terms. 
We refer to H\"ogele-da Costa \cite{Hog2} for a standard argument on such a estimate. One difference with  \cite{Hog2} is that there is an extra Brownian term here. Indeed, with the help of It\^o isometry, we obtain  
\begin{align}\label{Align-Sigma_2E}
\mathbb{E}\big[\sup_{ s\leqslant t\wedge \tau^\varepsilon} |\Sigma_2|^2\big] 
\leqslant 
\varepsilon c_4(p,d) \big(\sup_{ T^d\times D}|F_{I}|\big)^2 \int_0^{t}\mathbb{E}\big[| I_s^\varepsilon-I_s|^{2(p-1)}\big]ds.
\end{align}
Therefore, the estimate for (\ref{Align-Sigma0}) is quite similar to estimate (44) in \cite{Hog2} and yields a constant $c_5$ such that
\begin{align}\label{Align-Action}
\mathbb{E}\big[\sup_{ s\leqslant t\wedge \tau^\varepsilon} |I_s^{\varepsilon}-I_s|^p\big]
\leqslant& c_{5}\varepsilon^p(1+t^{2p+1}).
\end{align}
\\
{\bf Estimate of the angle component $|\theta_t^\varepsilon-\theta_t|$.} ~ For $s<\tau^\varepsilon$, 
applying the chain rule again, we have
\begin{align}
|\theta_s^\varepsilon-\theta_s|^p \leqslant
& p \int_0^s | \theta_r^\varepsilon-\theta_r|^{p-2} \left \langle \theta_r^\varepsilon-\theta_r,  \omega_0^i(I_r^\varepsilon)-\omega_0^i(I_r)\right \rangle dr \tag{$\Lambda_1$}\\
&+ p \sum_{k=1}^{d} \int_0^s | \theta_r^\varepsilon-\theta_r|^{p-2} \left \langle \theta_r^\varepsilon-\theta_r,  (\omega_k^i(I_r^\varepsilon)-\omega_k^i(I_r)) dB^k(r)\right \rangle  \tag{$\Lambda_2$}\\
&+ p \sum_{k=1}^{d} \int_0^s | \theta_r^\varepsilon-\theta_r|^{p-2} \left \langle \theta_r^\varepsilon-\theta_r,  (\omega_k^i(I_r^\varepsilon)-\omega_k^i(I_r)) dL^k(r)\right \rangle  \tag{$\Lambda_3$}\\
&+ p \sum_{k=1}^{d}\sum_{0\leqslant r\leqslant s} \int_0^s  | \theta_r^\varepsilon-\theta_r|^{p-1}\big|\phi(\Delta L^k (r), \omega_k^i(I_{r-}^\epsilon), I_{r-}^\epsilon) -\phi(\Delta L^k (r), \omega_k^i(I_{r-}), I_{r-})\notag\\
&- \big(I_{r-}^{\epsilon,i}-I_{r-}^i\big)-\big(\omega_k^i(I_{r-}^\epsilon)-\omega_k^i (I_{r-})\big)\Delta L^k (r) \big|  \tag{$\Lambda_4$}\notag\\
&+\varepsilon p  \int_0^s | \theta_r^\varepsilon-\theta_r|^{p-2} \Big| \left \langle \theta_r^\varepsilon-\theta_r,  K_\theta(\theta_t^\varepsilon,I_r^\varepsilon)\right \rangle \Big|dr \tag{$\Lambda_5$}\\
& + \varepsilon p\sum_{k=1}^d\int_0^s | \theta_r^\varepsilon-\theta_r|^{p-2} \big| \left \langle \theta_r^\varepsilon-\theta_r,   {F}_{\theta,k}(I_r^\varepsilon) d{B_r^k}\right \rangle\big| \tag{$\Lambda_6$}\\
& + \varepsilon p\sum_{k=1}^d\int_0^s | \theta_{r-}^\varepsilon-\theta_{r-}|^{p-2} \big| \left \langle \theta_{r-}^\varepsilon-\theta_{r-},   {G}_{\theta,k}(I_{r-}^\varepsilon) d{\tilde{L}_r^k} \right \rangle\big| \tag{$\Lambda_7$}\\
& + \varepsilon p \sum_{k=1}^d\sum_{0\leqslant r\leqslant s} \int_0^s | I_{r-}^\varepsilon-I_{r-}|^{p-1} dt
\big| \phi(\Delta L^k (r), {G}_{\theta,k}(I_{r-}^\varepsilon), I_{r-}^\varepsilon) \notag\\
&- I_{r-}^\varepsilon-{G}_{\theta,k}(I_{r-}^\varepsilon)\Delta L^k (r)\big| \tag{$\Lambda_8$}\\
=&\Lambda_1+\Lambda_2+\Lambda_3+\Lambda_4+\Lambda_5+\Lambda_6+\Lambda_7+\Lambda_8. \label{Align-Lambda}
\end{align}
Here we can replace the Stratonovitch integrations by It\^o integrations, as both $\omega_k(I)$ and $F_{\theta}(I)$ do not depend on $\theta$ and the Stratonovitch correction terms vanish. We next estimate each summand on the right hand side of equation above. Note that, for $k=0,1,2,...,d$,
\begin{equation}
|\omega_k(I_r^\varepsilon)-\omega_k(I_r)| \leqslant \sup_{ T^d\times D} |d{\omega}_k|\cdot |I_r^\varepsilon-I_r|. 
\end{equation}
The first term $\Lambda_1$ can be dealt with by Lipschitz estimate. Indeed, by Young's inequality and (\ref{Align-Action}), clearly we have
\begin{align}\label{Align-Lambda1E}
\mathbb{E}\big[\sup_{ s\leqslant t\wedge \tau^\varepsilon} \Lambda_1\big] 
&\leqslant c_{6}(p,d{\omega}_0) \mathbb{E}\Big[\int_0^{t\wedge \tau^\varepsilon} | \theta_s^\varepsilon-\theta_s|^{p-1}|I_s^\varepsilon-I_s| ds\Big]\notag\\
&\leqslant c_{6}\int_0^{t\wedge \tau^\varepsilon}  \mathbb{E}\big[\sup_{ r\leqslant s\wedge \tau^\varepsilon } | \theta_r^\varepsilon-\theta_r|^{p}|\big] ds+c_{13}\varepsilon^p t^{p+1}.
\end{align}
For the stochastic It\^o terms, we use the different kinds of maximal inequalities and the embedding $L^2\subset L^1$. 
It\^o isometry and yield
\begin{align}\label{Align-Lambda2E}
\mathbb{E}\big[\sup_{ s\leqslant t\wedge \tau^\varepsilon} \Lambda_2\big] 
&\leqslant c_{7}(p,d,d{\omega}_k) \mathbb{E}\Big[\int_0^{t\wedge \tau^\varepsilon} | \theta_s^\varepsilon-\theta_s|^{2p}|I_s^\varepsilon-I_s|^2 ds\Big]^{\frac{1}{2}}\notag\\
&\leqslant c_{7}\Big(\int_0^{t}  \mathbb{E}\big[\sup_{ r\leqslant s\wedge \tau^\varepsilon } | \theta_r^\varepsilon-\theta_r|^{p}|\big] ds\Big)^{\frac{1}{2}}+c_{15}\varepsilon^p t^{p+1}.
\end{align}
Kunita's first inequality (\cite[Page 265]{Ap}) with the exponent 2 yields
\begin{align}\label{Align-Lambda3E}
\mathbb{E}\big[\sup_{ s\leqslant t\wedge \tau^\varepsilon} \Lambda_3\big] 
\leqslant& c_{8}(p,d,d{\omega}_k) \Big(\mathbb{E}\Big[\int_0^{t\wedge \tau^\varepsilon} \int_{\mathbb{R}^d\setminus \{ 0 \}}| \theta_s^\varepsilon-\theta_s|^{2p}|I_s^\varepsilon-I_s|^2 |z|^2\nu(dz)ds\Big]^{\frac{1}{2}}\notag\\
&+ \int_{|z|> 1}| z|\nu(dz)\mathbb{E}\Big[\int_0^{t\wedge \tau^\varepsilon} | \theta_s^\varepsilon-\theta_s|^{p-1}|I_s^\varepsilon-I_s| ds\Big]\Big)\notag\\
\leqslant& c_{9}\bigg( \Big(\int_0^{t}  \mathbb{E}\big[\sup_{ r\leqslant s\wedge \tau^\varepsilon } | \theta_r^\varepsilon-\theta_r|^{p}|\big] ds\Big)^{\frac{1}{2}}+\int_0^{t}  \mathbb{E}\big[\sup_{ r\leqslant s\wedge \tau^\varepsilon } | \theta_r^\varepsilon-\theta_r|^{p}|\big] ds + \varepsilon^p t^{p+1} \bigg)
\end{align}
For canonical Marcus terms, we adapt the methods developed in \cite[Section 3]{Hog2}. In fact, the term $\Lambda_4$ can be estimated in terms of the quadratic variation of $L_t$ as shown in \cite{Hog2}. We rewrite the result in terms of the compensated Poisson random measure $\tilde{N}$ and then obtain
\begin{align}\label{Align-Lambda4E}
\mathbb{E}\big[\sup_{ s\leqslant t\wedge \tau^\varepsilon} \Lambda_4\big] 
\leqslant& c_{10} \mathbb{E}\Big[\sup_{ s\leqslant t\wedge \tau^\varepsilon} \sum_{k=1}^d \sum_{0<s\leqslant t} | \theta_s^\varepsilon-\theta_s|^{p-1}|I_s^\varepsilon-I_s| |\triangle L^k(s)|^2\Big]\notag\\
\leqslant& c_{11}\Big(\int_{\mathbb{R}^d\setminus \{ 0 \}}|z|^4\nu(dz)\int_0^{t\wedge \tau^\varepsilon}  \mathbb{E}\big[\sup_{ r\leqslant s\wedge \tau^\varepsilon } | \theta_r^\varepsilon-\theta_r|^{2(p-1)}|I_r^\varepsilon-I_r|^2\big] ds\Big)^{\frac{1}{2}} \notag\\
&+\int_0^{t\wedge \tau^\varepsilon}  \mathbb{E}\big[\sup_{ r\leqslant s\wedge \tau^\varepsilon } | \theta_r^\varepsilon-\theta_r|^{p-1}|\big] ds\notag\\
\leqslant&  c_{12}\bigg( \Big(\int_0^{t}  \mathbb{E}\big[\sup_{ r\leqslant s\wedge \tau^\varepsilon } | \theta_r^\varepsilon-\theta_r|^{p}|\big] ds\Big)^{\frac{1}{2}}+\int_0^{t}  \mathbb{E}\big[\sup_{ r\leqslant s\wedge \tau^\varepsilon } | \theta_r^\varepsilon-\theta_r|^{p}|\big] ds+\varepsilon^p t^{p+1}\bigg)
\end{align}
Observe that the terms $\Lambda_5$ - $\Lambda_8$ are structurally identical to $\Sigma_1$ - $\Sigma_4$ and they can be estimated analogously by replacing $K_I+\frac{d}{2}\hat{D}_IF_I$, $\tilde{K}_I$ and $\tilde{G}_I$ by $K_\theta$, $\tilde{K}_\theta$ and $\tilde{G}_\theta$, respectively. Hence
\begin{align}\label{Align-Lambda5E}
\mathbb{E}\big[\sup_{ s\leqslant t\wedge \tau^\varepsilon} \Lambda_5\big] 
&\leqslant c_{13} \varepsilon\mathbb{E}\Big[\int_0^{t\wedge \tau^\varepsilon} | \theta_s^\varepsilon-\theta_s|^{p-1}|ds\Big]\notag\\
&\leqslant c_{14} \varepsilon\int_0^{t}  \mathbb{E}\big[\sup_{ r\leqslant s\wedge \tau^\varepsilon } | \theta_r^\varepsilon-\theta_r|^{p}|\big] ds+c_{14}\varepsilon^p t,
\end{align}
\begin{align}\label{Align-Lambda6E}
\mathbb{E}\big[\sup_{ s\leqslant t\wedge \tau^\varepsilon} \Lambda_6\big] 
&\leqslant c_{15} \varepsilon \mathbb{E}\Big[\int_0^{t\wedge \tau^\varepsilon} | \theta_s^\varepsilon-\theta_s|^{2(p-1)} ds\Big]^{\frac{1}{2}}\notag\\
&\leqslant c_{16} \varepsilon \Big(\int_0^{t}  \mathbb{E}\big[\sup_{ r\leqslant s\wedge \tau^\varepsilon } | \theta_r^\varepsilon-\theta_r|^{p}|\big] ds\Big)^{\frac{1}{2}}+c_{16}\varepsilon^p t,
\end{align}
and
\begin{align}\label{Align-Lambda78E}
\mathbb{E}\big[\sup_{ s\leqslant t\wedge \tau^\varepsilon} (\Lambda_7+\Lambda_8)\big] 
\leqslant& c_{17} \varepsilon \Big(\mathbb{E}\Big[\int_0^{t\wedge \tau^\varepsilon} | \theta_s^\varepsilon-\theta_s|^{2(p-1)} ds\Big]^{\frac{1}{2}}
+ \mathbb{E}\Big[\int_0^{t\wedge \tau^\varepsilon} | \theta_s^\varepsilon-\theta_s|^{p-1}|ds\Big]\Big)\notag\\
\leqslant&  c_{18}  \bigg( \Big(\int_0^{t}  \mathbb{E}\big[\sup_{ r\leqslant s\wedge \tau^\varepsilon } | \theta_s^\varepsilon-\theta_s|^{p}|\big] ds\Big)^{\frac{1}{2}}+\int_0^{t}  \mathbb{E}\big[\sup_{ r\leqslant s\wedge \tau^\varepsilon } | \theta_r^\varepsilon-\theta_r|^{p}|\big] ds+\varepsilon^p t\bigg).
\end{align}
Taking the supremum and expectation in inequality (\ref{Align-Lambda}) and combining the estimates (\ref{Align-Lambda1E}) - (\ref{Align-Lambda78E}), we obtain
\begin{align}\label{Align-Angle}
\mathbb{E}\big[\sup_{ s\leqslant t\wedge \tau^\varepsilon}|\theta_s^\varepsilon-\theta_s|^p\big]
\leqslant&  c_{19}  \bigg( \Big(\int_0^{t}  \mathbb{E}\big[\sup_{ r\leqslant s\wedge \tau^\varepsilon } | \theta_r^\varepsilon-\theta_r|^{p}|\big] ds\Big)^{\frac{1}{2}}+\int_0^{t}  \mathbb{E}\big[\sup_{ r\leqslant s\wedge \tau^\varepsilon } | \theta_r^\varepsilon-\theta_r|^{p}|\big] ds\notag\\
&+t\varepsilon^p(1+t^p)\bigg).
\end{align}
That is, for $u(t):=\mathbb{E}\big[\sup_{ s\wedge \tau^\varepsilon \leqslant t}|\theta_s^\varepsilon-\theta_s|^p\big]$, $p(t)=t\varepsilon^p(1+t^p)$ and the concave invertible function $f(x)=\sqrt{x}+x$ we have
\begin{align}\label{Align-Angle-0}
u(t)\leqslant c_{19} f\Big(\int_0^{t} u(s) ds\Big)+c_{19}p(t).
\end{align}
By the nonlinear extension of the Gronwall-Bihari inequality (see \cite{PBG}), or a nonlinear comparison principle in \cite{Hog2}, we finally have
\begin{align}\label{Align-Angle-1}
\mathbb{E}\big[\sup_{ s\leqslant t\wedge \tau^\varepsilon}|\theta_s^\varepsilon-\theta_s|^p\big]\leqslant c_{20}\varepsilon^{2p}\exp({c_{31}t}).
\end{align}
\\
Eventually, the desired result follows from Minkowski's inequality and the estimates (\ref{Align-Action}) and (\ref{Align-Angle-1}),
\begin{eqnarray}
\left[
\mathbb{E}(\sup_{s\leqslant t\wedge \tau^\varepsilon} | f(Y_s^\varepsilon)-f(X_s) |^p)
\right]^{\frac{1}{p}}
\leqslant C_1\varepsilon\exp(C_2 t).
\end{eqnarray}
\end{proof}

This lemma shows that, over a time interval $t$, the first integrals of the perturbed system change by an order $\varepsilon \exp(C_2 t)$, and the slow component thus accumulate over a time interval of the size $t/\varepsilon$. Next, we would like to show that the randomness in the fast component could be averaged out by the induced invariant measure, and we can obtain a new dynamical system as $\varepsilon$ goes to zero. 
\par
For convenience, we adopt the following notation. Let $g: M\to \mathbb{R}$ be a continuous function, and $\tilde{g}: \mathbb{T}^d\times D \to \mathbb{R}$ be its representation in action-angle coordinate. We define the average of $g$ over the torus as $Q^g: D\subset\mathbb{R}^d\to\mathbb{R}$, i.e.,
\begin{equation}
Q^g(h)=\int_{ \mathbb{T}^d}\tilde{g}(h,z)\mu(dz)
\end{equation}
We remark that this average can be also understood in the sense of $\mu_h$ by taking the canonical transformation map $\pi^{\prime}: M_h\to \mathbb{T}^d$. Indeed, the induced measure $\pi^{\prime}(\mu_h)$ is the Lebesque measure $\mu$ on the torus and the average can be written as $Q^g(h)=\int_{M_h}g(h,z)\mu_h(dz)$ formally.

\begin{lemma}\label{lemma 4.5}
(Estimation of the averaging error) Suppose that $g$ is continuous on $U_0$. Set $H_i^{\varepsilon}(s)=H_i(Y_{t/\varepsilon}^{\varepsilon})$ and
$H^\varepsilon(s)=(H_1^\varepsilon(s),...,H_d^\varepsilon(s))$. For $\tau^\varepsilon=\inf \{ t\geqslant0:Y_{t/\varepsilon}^{\varepsilon}\notin U_0\} $, we denote by
\begin{equation}\label{Eqnarray-error}
\delta^g(\varepsilon,t)=\int_{s\wedge \tau^\varepsilon}^{(s+t)\wedge \tau^\varepsilon}g(Y_{r/\varepsilon}^{\varepsilon})dr
-\int_{s\wedge (\tau^\varepsilon/\varepsilon)}^{(s+t)\wedge (\tau^\varepsilon/\varepsilon)}Q^g(H^\varepsilon(r))dr
\end{equation}
the averaging error. Then,  for any given $t>0$ and sufficiently small $\varepsilon>0$, there are constants $k_1,k_2>0$ such that 
\begin{equation}\label{Eqnarray-error-est}
\big(\mathbb{E}[\sup_{s\leqslant t} |\delta^g(\varepsilon,s)|^p\big)^{\frac{1}{p}}\leqslant k_1 t \big(\varepsilon^{1-k_2 t} +\eta(t|\ln\varepsilon|)\big).
\end{equation}
where $\eta(t)$ is the rate of convergence for ergodicity assumption ${\bf A6}$.
\end{lemma}
\begin{proof}
The main idea is to use the approximate result in Lemma (\ref{lemma 4.4}) on sufficiently small intervals and to apply the ergodicity assumption to replace time average by space average. We refer to Li \cite{Li} for a nice proof in the Brownian case and H\"ogele-Ruffino \cite{Hog}, Gargate-Ruffino \cite{GGI} and H\"ogele-da Costa \cite{Hog2} for the extensions of this proof method. For sufficiently small $\varepsilon>0$ and $t\geqslant 0$ we define the partition
$$
t_0=0<t_1<...<t_{N_\varepsilon}\leqslant \frac{t}{\varepsilon}\wedge \tau^\varepsilon
$$
with the following assignment of increments:
$$
\triangle_\epsilon t =t|\ln\varepsilon|.
$$
The grid points of the partition are given by $t_n^\varepsilon=n\triangle_\varepsilon t$ for $0\leqslant n\leqslant N_\varepsilon$ with $N_\varepsilon=\lfloor (\varepsilon|\ln\varepsilon|)^{-1}\rfloor$ where the bracket function $\lfloor \cdot \rfloor$ denotes the integer part of the value.
\par
Initialy we represent the left hand side as the sum:
\begin{equation}
\int_{0}^{t\wedge \varepsilon\tau^\varepsilon}g(Y_{t/\varepsilon}^{\varepsilon})dr=\varepsilon\int_{0}^{\frac{t}{\varepsilon}\wedge \tau^\varepsilon}g(Y_{r}^\varepsilon)dr
=\varepsilon\sum_{n=0}^{N_\varepsilon-1}\int_{t_n}^{t_{n+1}}g(Y_{r}^\varepsilon)dr+\varepsilon\int_{t_n}^{\frac{t}{\varepsilon}\wedge \tau^\varepsilon}g(Y_{r}^\varepsilon)dr.
\end{equation}
\par
Suppose that $\Psi:=\Psi_t=(\Psi(t,\omega,x),t\in\mathbb{R}^{+})$ the solution flow of the unperturbed stochastic differential equation (\ref{Equation-H}) with initial point $x$ and $\Theta_t$ the shift operator on the canonical probability space, i.e., $\Theta_t(\omega)(-)=\omega(-+t)-\omega(t)$. Then,
\begin{align}\label{Eqnarray-Pi}
|\delta^g(\varepsilon,t)|
\leqslant&\varepsilon|\sum_{n=0}^{N_\varepsilon-1}\int_{t_n}^{t_{n+1}}g(Y_{r}^\varepsilon)dr-g(\Psi_{r-t_n}(\Theta_{t_n}(\omega),Y_{r}^\varepsilon))dr|\notag\\
&+\varepsilon|\sum_{n=0}^{N_\varepsilon-1}\int_{t_n}^{t_{n+1}}g(\Psi_{r-t_n}(\Theta_{t_n}(\omega),Y_{r}^\varepsilon))-\triangle_\varepsilon t Q^g(H^\varepsilon(\varepsilon t_n))dr|\notag\\
&+\varepsilon|\sum_{n=0}^{N_\varepsilon-1}\triangle_\varepsilon t Q^g(H^\varepsilon(\varepsilon t_n))-\int_0^{t\wedge \varepsilon\tau^\varepsilon}Q^g(H^\varepsilon(r))dr|\notag\\
&+\varepsilon|\int_{t_n}^{\frac{t}{\varepsilon}\wedge \tau^\varepsilon}g(Y_{r}^\varepsilon)dr|\notag\\
=&\Pi_1+\Pi_2+\Pi_3+\Pi_4.
\end{align}
\par
We proceed showing that the preceding four terms tend to zero uniformly on compact intervals. In the proof below $c$ stands for an unspecified constant. Using the Markov property, Lemma 4.8 and H\"older's inequality,
\begin{align}\label{Eqnarray-e1}
\big(\mathbb{E}\sup_{s\leqslant t}|\Sigma_1|^p\big)^\frac{1}{p}
\leqslant&\varepsilon\sum_{n=0}^{N_\varepsilon-1}\Big(\mathbb{E}\big[\int_{t_n}^{t_{n+1}}\sup_{{t_n}\leqslant s\leqslant r}|g(Y_{r}^\varepsilon)dr-g(\Psi_{r-t_n}(\Theta_{t_n}(\omega),Y_{r}^\varepsilon))|dr\big]^p\Big)^{\frac{1}{p}}\notag\\
\leqslant&\varepsilon\sum_{n=0}^{N_\varepsilon-1}(\triangle_\varepsilon t)^{\frac{p-1}{p}}\Big(\mathbb{E}\big[\int_{t_n}^{t_{n+1}}\sup_{{t_n}\leqslant s\leqslant r}|g(Y_{r}^\varepsilon)dr-g(\Psi_{r-t_n}(\Theta_{t_n}(\omega),Y_{r}^\varepsilon))|^pdr\big]\Big)^{\frac{1}{p}}\notag\\
\leqslant&\varepsilon N_\varepsilon (\triangle_\varepsilon t) C_1\varepsilon e^{C_2 \triangle_\varepsilon t}
=\varepsilon \lfloor (\varepsilon|\ln\varepsilon|)^{-1}\rfloor\cdot t|\ln\varepsilon| \cdot C_1 \varepsilon e^{- C_2 t\ln\varepsilon }\notag\\
\leqslant& c t \varepsilon^{1-C_2 t}
\end{align}
\par
We denote $\mu_{H^\varepsilon(\varepsilon t_n)}$ by the invariant measure on the invariant manifold $M_{H^\varepsilon(\varepsilon t_n)}\equiv M_{Y_{t_n}^\varepsilon}$.
Ergodicity assumption ${\bf A6}$ and the Markov property of the flow yield,
\begin{align}\label{Eqnarray-e2}
\big(\mathbb{E}\sup_{s\leqslant t}|\Sigma_2|^p\big)^\frac{1}{p}
\leqslant&\varepsilon\sum_{n=0}^{N_\varepsilon-1}\Big(\mathbb{E}\sup_{s\leqslant t}|\int_{t_n}^{t_{n+1}}g(\Psi_{r-t_n}(\Theta_{t_n}(\omega),Y_{r}^\varepsilon))dr-\triangle_\varepsilon t Q^g(H^\varepsilon(\varepsilon t_n))|^p\Big)^{\frac{1}{p}}\notag\\
\leqslant&\varepsilon\triangle_\varepsilon t\sum_{n=0}^{N_\varepsilon-1}\Big(\mathbb{E}\sup_{s\leqslant t}|\frac{1}{\triangle_\varepsilon t}\int_{t_n}^{t_n+\triangle_\varepsilon t}g(\Psi_{r-t_n}(\Theta_{t_n}(\omega),Y_{r}^\varepsilon))dr
-Q^g(H^\varepsilon(\varepsilon t_n))|^p\Big)^{\frac{1}{p}}\notag\\
\leqslant&\varepsilon N_\varepsilon\triangle_\varepsilon t \sup_n\Big(\mathbb{E}\sup_{s\leqslant t}|\frac{1}{\triangle_\varepsilon t}\int_{t_n}^{t_n+\triangle_\varepsilon t}g(\Psi_{r-t_n}(\Theta_{t_n}(\omega),Y_{r}^\varepsilon))dr\notag\\
&-\int_{M_{H^\varepsilon(\varepsilon t_n)}}g(H^\varepsilon(\varepsilon t_n),z)d\mu_{H^\varepsilon(\varepsilon t_n)}(z)|^p\Big)^{\frac{1}{p}}\notag\\
\leqslant&c\varepsilon N_\varepsilon\triangle_\varepsilon t   \eta(\triangle_\varepsilon t) 
=\varepsilon \lfloor (\varepsilon|\ln\varepsilon|)^{-1}\rfloor\cdot t|\ln\varepsilon| \cdot \eta(t|\ln\varepsilon|)\notag\\
\leqslant&ct\eta(t|\ln\varepsilon|).
\end{align}
\par
Note that $g$ is $C^1$ on $U_0$, both $\sup_{U_0}|g|$ and $\sup_{U_0}|dg|$ are finity. We have the following estimates:
\begin{align}\label{Eqnarray-e3}
\big(\mathbb{E}\sup_{s\leqslant t}|\Sigma_3|^p\big)^\frac{1}{p}
\leqslant&\varepsilon\sum_{n=0}^{N_\varepsilon-1}\triangle_\varepsilon t(\mathbb{E}\sup_{\varepsilon t_n\leqslant r\leqslant \varepsilon t_n+1}|Q^g(H^\varepsilon(\varepsilon t_n))-Q^g(H^\varepsilon(\varepsilon r))|^p)^\frac{1}{p}\notag\\
\leqslant&c\varepsilon N_\varepsilon \triangle_\varepsilon t \cdot c\varepsilon \exp(c \triangle_\varepsilon t)
  \notag\\
\leqslant& c t \varepsilon^{1-c t},
\end{align}
\par
and 
\begin{eqnarray}\label{Eqnarray-e4}
\big(\mathbb{E}\sup_{s\leqslant t}|\Sigma_4|^p\big)^\frac{1}{p}\leqslant c \varepsilon t|\ln\varepsilon|.
\end{eqnarray}
\par
Consequently, the desired result follows from inequality (\ref{Eqnarray-Pi}), estimates (\ref{Eqnarray-e1}) - (\ref{Eqnarray-e4}), and Minkowski's inequality.
\end{proof}

At last, we present the proof of Theorem \ref{theorem 4.1} based the results of Lemma \ref{lemma 4.4} and Lemma \ref{lemma 4.5}.

\begin{proof}[\textbf{Proof of Theorem \ref{theorem 4.1}.}]
Applying the change of variable formula \cite{KPP} for Marcus SDE (\ref{Eqnarray-perturbation}) and using the completely integrability assumption {\bf A3}, we have for $t<\tau_0\wedge \tau^\varepsilon$, $1\leqslant i \leqslant d$,
\begin{align}
H_i^\varepsilon(t)
=&H_i(Y_0)+\int_0^t\omega^2(V_{i}, K)(Y_{s/\varepsilon}^{\varepsilon})ds \notag\\
&+ \int_0^t\sum_{k=1}^d \omega^2(V_{i}, F_k\circ \pi)(Y_{s/\varepsilon}^{\varepsilon})\circ d\tilde{B}_t^k
+\int_0^t \sum_{k=1}^d \omega^2(V_{i}, G_k\circ \pi)(Y_{s/\varepsilon}^{\varepsilon})\diamond d\tilde{L}_t^k.
\end{align}

For $i$ fixed, we write
\begin{equation}
g_i=\omega^2(V_{i}, K)
\end{equation}
which is $C^1$ on $U_0$. 
Applying (\ref{Eqnarray-error}) to the functions $g_i$, we obtain for any $t<\tau^\varepsilon$,
\begin{equation}
\int_{0}^{t\wedge \tau^\varepsilon}g_i(Y_{s/\varepsilon}^{\varepsilon})ds=\int_{0}^{t\wedge (\tau^\varepsilon/\varepsilon)}Q^{g_i}(H^\varepsilon(s))ds+\delta^{g_i}(\varepsilon,t).
\end{equation}
On the other hand, using the notations of the previous two lemmas, the equation (\ref{Eqnarray-aver}) can be written as
\begin{align}
&{d}\bar{H}_i(t)=Q^{g_i}(\bar{H}_i(t))dt+ \sum_{k=1}^d F_{I,k}(\bar{H}_i(t))\circ d\tilde{B}_t^k+\sum_{k=1}^d G_{I,k}(\bar{H}_i(t))\diamond d\tilde{L}_t^k,\notag\\
&\bar{H}_0(t)=H(Y_0) .\notag
\end{align}
Therefore, for any $t<\tau^\varepsilon$, we have
\begin{align}
|H_i^\varepsilon(t\wedge \tau^\varepsilon)-\bar{H}_i(t\wedge \tau^\varepsilon)|
\leqslant& \int_0^{t\wedge \tau^\varepsilon}|Q^{g_i}(H_i^\varepsilon(s))-Q^{g_i}(\bar{H}_i(s))|ds+\delta(g_i, \varepsilon, t)\notag\\
+& \int_0^{t\wedge \tau^\varepsilon} |F_{I,k}(H_i^\varepsilon(s))-F_{I,k}(\bar{H}_i(s))|\circ d\tilde{B}_t^k \notag\\
+& \int_0^{t\wedge \tau^\varepsilon} |G_{I,k}(H_i^\varepsilon(s))-G_{I,k}(\bar{H}_i(s))|\diamond d\tilde{L}_t^k 
\end{align}
\par
Note that the estimate of the first term is straight forward Lipschitz estimate,
\begin{eqnarray}
 \int_0^{t\wedge \tau^\varepsilon}|Q^{g_i}(H_i^\varepsilon(s))-Q^{g_i}(\bar{H}_i(s))|ds
\leqslant C(g,\varphi)\int_0^{t\wedge \tau^\varepsilon}|H_i^\varepsilon(s))-\bar{H}_i(s)|ds.
\end{eqnarray}
The estimates of the Brownian term and the L\'evy term can be dealt with by Kunita's second inequality \cite[Page 268]{Ap} or other maximal inequalities. The computation for these two terms are very similar to that in the proof of Lemma \ref{lemma 4.4}, and we refer to \cite[Section 5]{Hog2} for a detailed procedure. Finally,
\begin{eqnarray}
 \mathbb{E}\big[ \sup_{s\leqslant t\wedge \tau^\varepsilon} |H_i^\varepsilon(s)-\bar{H}_i(s)|^p\big]
\leqslant C_1 \int_{0}^{t} \mathbb{E}\big[ \sup_{r\leqslant s\wedge \tau^\varepsilon} |H_i^\varepsilon(r)-\bar{H}_i(r)|^p\big]ds+ \mathbb{E}\big[\sup_{s\leqslant t}|\delta(g_i, \varepsilon, t)|^p\big] \notag\\
\end{eqnarray}
\par
By Lemma \ref{lemma 4.5} and Gronwall's inequality, there is a constant $k_3>0$ such that
\begin{align}
\bigg( \mathbb{E}\big[ \sup_{s\leqslant t\wedge \tau^\varepsilon} |H_i^\varepsilon(s)-\bar{H}_i(s)|^p\big] \bigg)^\frac{1}{p}
\leqslant&  \mathbb{E}\big[\sup_{s\leqslant t}|\delta(g_i, \varepsilon, t)|^p\big]^{^\frac{1}{p}}\exp(k_3 t) \notag\\
\leqslant& k_1 t \big(\varepsilon^{1-k_2 t} +\eta(t|\ln\varepsilon|)\big)\exp(k_3 t).
\end{align}

For the second part of the theorem, we have the following estimate by the definition of $\tau^\varepsilon$, $\tau_\delta$ and Chebychev's inequality,
\begin{eqnarray}
\mathbb{P}(\tau^\varepsilon<\tau_\delta)
&\leqslant& \mathbb{P}(\sup_{s\leqslant \tau^\varepsilon\wedge\tau_\delta} |\bar{H}(s)-H^\varepsilon(s)|>\delta) \notag\\
&\leqslant&\delta^{-p}\mathbb{E}\big[ \sup_{s\leqslant \tau^\varepsilon\wedge\tau_\delta}|\bar{H}(s)-H^\varepsilon(s)|^p\big] \notag\\
&\leqslant& k_4\delta^{-p} t^p \big(\varepsilon^{1-k_2 t} +\eta(t|\ln\varepsilon|)\big)^p\exp(k_5 t).
\end{eqnarray}
\end{proof}

\subsection{An example: Perturbed stochastic harmonic oscillator with L\'{e}vy noise} \label{section 4.4}
\noindent
In this subsection, let's present a simple illustrative example for the above averaging principle of integrable stochastic Hamiltonian system with L\'{e}vy noise. We write $(q, p)= (q_1, . . . ,q_d, p_1, . . . , p_d)$ as canonical coordinates, and there is an important class of Hamiltonian functions on $\mathbb{R}^{2n}$ of the form $H(q,p)=\frac{1}{2}|p|^2+V(q)$, i.e. Hamiltonian $H$ is the sum of kinetic, $T =\frac{1}{2}|p|^2=\frac{1}{2}\sum_{i=1}^dp_i^2$ and potential, $V (q)$, energies. Furthermore, if $V$ is quadratic, e.g. $V (q)=\frac{1}{2}\varpi|q|^2$ with $\varpi$ a frequency,  then we have the linear harmonic oscillator. Given Hamiltonian functions as follow,
\begin{eqnarray}
H_1&=&\frac{1}{2}\sum_{i=1}^dp_i^2+\frac{1}{2}\sum_{i=1}^d\varpi_i^2p_i^2, \notag\\
H_k&=&\frac{1}{2}\frac{p_k^2}{\varpi_k}+\frac{1}{2}\varpi_kp_k^2, ~~k=2,...,d,\notag
\end{eqnarray}
and a smooth function $H_0$ commuting with all $H_k$, $k=1,...,d$, i.e.
$$\{H_0, H_k\}=\sum_{i=1}^d\Big(  \frac{\partial H_0}{\partial p_i}\frac{\partial H_k}{\partial q_i} - \frac{\partial H_0}{\partial q_i}\frac{\partial H_k}{\partial p_i}  \Big)=0,$$
we have
\begin{eqnarray}
dq_i(t)&=&\frac{\partial H_0}{\partial p_i}dt+\sum_{i=1}^d\frac{\partial H_k}{\partial p_i}\circ dB_t^k+\sum_{i=1}^d\frac{\partial H_k}{\partial p_i}\diamond d{L}_t^k, \\
dp_i(t)&=&-\frac{\partial H_0}{\partial q_i}dt-\sum_{i=1}^d\frac{\partial H_k}{\partial q_i}\circ dB_t^k-\sum_{i=1}^d\frac{\partial H_k}{\partial q_i}\diamond d{L}_t^k,
\end{eqnarray}
which is an integrable stochastic Hamiltonian system with  $\alpha$-stable L\'{e}vy noise. Let $J=\begin{bmatrix}\begin{smallmatrix}0 & I \\ -I &  0\end{smallmatrix}\end{bmatrix}$ be a $2d \times 2d$ antisymmetric matrix, which is called Poisson matrix, this system is equivalent to
\begin{eqnarray}\label{Eqnarray-example}
dX_t=J\nabla H_0(X_t)dt + \sum_{i=1}^dJ\nabla H_k(X_t)\circ dB_t^k +  \sum_{i=1}^dJ\nabla H_k(X_t)\diamond d{L}_t^k.
\end{eqnarray}
For $M_h =\{x \in M: H_k(x) = h_k, ~k=1,2,..d\}$,
if we take an action-angle coordinates change $\varphi^{-1}: U_0 \to \mathbb{T}^d\times D, (q,p)\mapsto (\theta,I)$,
\begin{eqnarray}\label{Example-AA}
q_i=\sqrt{\frac{2I_i}{\varpi_i}}\cos \theta_i,~p_i=\sqrt{2\varpi_iI_i}\sin \theta_i,
\end{eqnarray}
then the induced Hamiltonians $
H_k^{\prime}=H_k(\varphi(\theta,I))=\left\{
\begin{array}{rl}
\sum_{i=1}^d \varpi_iI_i, ~ k=1\\
I_k,~k=2,...,d
\end{array}
\right.
$ on $\mathbb{T}^d\times D$ satisfy,
\begin{eqnarray}
\dot{\theta}_k^i&=&\frac{\partial H_k^{\prime}}{\partial I_i}=:\omega_k^i(I)=
\left\{
\begin{array}{rl}
&\varpi_i, ~ k=1;\\
&1,~~~k=2,...,d~ \& ~i=k;\\
&0, ~~~otherwise.\textsc{}
\end{array}
\right. \notag\\
\dot{I}_k^i&=&-\frac{\partial H_k^{\prime}}{\partial \theta_i}=0.\notag
\end{eqnarray}
Next, we investigate the effective behavior of a small transversal perturbation of order $\varepsilon$ to this system. For simplicity, we consider the case on $\mathbb{R}^4$ with $\varpi =1$ and $L_t$ having second moments,
\begin{eqnarray}\label{Example-U}
d\begin{pmatrix}q_1 \\ q_2 \\ p_1\\p_2\end{pmatrix}
=\begin{pmatrix}p_1 &0\\ p_2&p_2 \\ -q_1&0\\-q_2&-q_2\end{pmatrix}\circ d\begin{pmatrix}B_t^1 \\ B_t^2\end{pmatrix}
+\begin{pmatrix}p_1 &0\\ p_2&p_2 \\ -q_1&0\\-q_2&-q_2\end{pmatrix}\diamond d\begin{pmatrix}{L}_t^1 \\ {L}_t^2\end{pmatrix}.
\end{eqnarray}
Take the perturbation vectors to be $\varepsilon K=\big(0,\varepsilon q_2/(q_2^2+p_2^2),0,0 \big)^T$, $\varepsilon(E, O)^T\dot{\tilde{B}}_t$ and $\varepsilon(E, O)^T\dot{\tilde{L}}_t$, where $E$ is the  identity matrix, $O$ is the zero matrix and $\tilde{L}_t$ is a pure jump L\'evy motion with four-order moments. 
By action-angle coordinates change (\ref{Example-AA}), we have, with $\Lambda=\begin{bmatrix}\begin{smallmatrix}1 & 1 \\ 1 &  0\end{smallmatrix}\end{bmatrix}$,
$$
d\begin{pmatrix}\theta_1 \\ \theta_2 \\ I_1\\I_2\end{pmatrix}
=\begin{pmatrix}\Lambda \\ O\end{pmatrix} d\begin{pmatrix}B_t^1 \\ B_t^2\end{pmatrix}
+\begin{pmatrix}\Lambda \\ O\end{pmatrix} d\begin{pmatrix}{L}_t^1 \\ {L}_t^2\end{pmatrix}
+\varepsilon\begin{pmatrix}0 \\ \frac{1}{2I_2}\sin\theta_2\cos\theta_2 \\ 0\\-\cos^2\theta_2\end{pmatrix}dt+\varepsilon \begin{pmatrix}E \\ O\end{pmatrix} d\begin{pmatrix}\tilde{B}_t^1 \\ \tilde{B}_t^2\end{pmatrix}+\varepsilon \begin{pmatrix}E \\ O\end{pmatrix} d\begin{pmatrix}\tilde{L}_t^1 \\ \tilde{L}_t^2\end{pmatrix}.
$$
For unperturbed system, it is easy to get fundamental solution with initial condition $(q_0,p_0)=\varphi(\theta_0,I_0)$: $q_t=\sqrt{2I_0}\cos(\Lambda(B_t+L_t)), p_t=\sqrt{2I_0}\sin(\Lambda(B_t+L_t))$ with $\Lambda=\begin{bmatrix}\begin{smallmatrix}1 & 1 \\ 1 &  0\end{smallmatrix}\end{bmatrix}$. 
Note that
\begin{eqnarray}
g_i:=\omega^2(V_i, K)=V_i^TJK
=\frac{q_2^2}{q_2^2+p_2^2} ~\Longrightarrow~ \tilde{g}_i=\cos^2 \theta_2 ,~i=1,2.\notag
\end{eqnarray}
We obtain
$$
Q^g(h_i)=\frac{1}{(2\pi)^2}\int_0^{2\pi}\int_0^{2\pi} \cos^2 \theta_2 d\theta_1d\theta_2=\frac{1}{2}.
$$
We verify that
$\frac{1}{t} \int_0^t g_i (q_s,p_s)ds \to Q^g(h_i)$ in $L^2$, as $t \to \infty$, with a rate of convergence $\eta(t)=\frac{1}{\sqrt{t}}$ in the Appendix.
Therefore, the transversal system stated in Theorem \ref{theorem 4.1} is $\bar{H}_i(t)=\frac{t}{2}$. The result guarantees that, on the accelerated time scale $\frac{t}{\varepsilon}$, $H_i^\varepsilon(t)$ has a local behavior close to $\frac{t}{2}$ in the sense that
\begin{eqnarray}
\bigg( \mathbb{E}\Big[ \sup_{s\leqslant t\wedge \tau^\varepsilon} \big|H_i^\varepsilon(s)-\frac{t}{2}\big|^2\Big] \bigg)^\frac{1}{2}
\leqslant k_1 t \big(\varepsilon^{1-k_2 t} +\frac{t}{2}|\ln\varepsilon|\big)\exp(k_3 t)
\end{eqnarray}
tends to $0$ when $\varepsilon\to 0$, for any fixed $t$ and the constant $k_1,k_2,k_3>0$.

\section*{Appendix: Proofs of Theorem \ref{theorem 3.1} -\ref{theorem 3.2} and Calculations of Example \ref{section 4.4}}

We now prove theorem \ref{theorem 3.1} and theorem \ref{theorem 3.2} which are based on the formula of change of variables in differential forms.

\begin{proof}[\textbf{Proof of Theorem \ref{theorem 3.1}.}] \emph{\textbf{}}
\\
Noticing that
\begin{eqnarray}
dP\wedge dQ &=&\sum_{i=1}^{n}dP^i\wedge dQ^i \notag\\
&=&\sum_{i=1}^{n}\sum_{l=r+1}^{n}\sum_{r=1}^{n}\left[
(\frac{\partial P^i}{\partial p^r} \frac{\partial Q^i}{\partial p^l} - \frac{\partial P^i}{\partial p^l} \frac{\partial Q^i}{\partial p^r})dp^r\wedge dp^l+
(\frac{\partial P^i}{\partial q^r} \frac{\partial Q^i}{\partial q^l} - \frac{\partial P^i}{\partial q^l} \frac{\partial Q^i}{\partial q^r})dq^r\wedge dq^l
\right]\notag\\
&& +\sum_{i=1}^{n}\sum_{l=1}^{n}\sum_{r=1}^{n}(\frac{\partial P^i}{\partial p^r} \frac{\partial Q^i}{\partial q^l} - \frac{\partial P^i}{\partial q^l} \frac{\partial Q^i}{\partial p^r})dp^r\wedge dq^l,\notag
\end{eqnarray}
we infer that the phase flow of (\ref{Equation-q1} - \ref{Equation-p1}) preserves symplectic structure if and only if
\begin{equation}\label{Equation-QP0}
\left\{
\begin{array}{rl}
\sum_{i=1}^{n}\frac{D(P^i,Q^i)}{D(p^r,p^l)}=0, ~~ r\neq l, \\
\sum_{i=1}^{n}\frac{D(P^i,Q^i)}{D(q^r,q^l)}=0, ~~ r\neq l, \\
\sum_{i=1}^{n}\frac{D(P^i,Q^i)}{D(p^r,q^l)}=\delta_{rl}, ~r,l=1,...,n. \\
\end{array}
\right.
\end{equation}
Clearly,
$$\frac{D(P^i(t_0),Q^i(t_0))}{D(p^r,p^l)}=\frac{D(p^i,q^i)}{D(p^r,p^l)}=0, \; \;  \frac{D(P^i(t_0),Q^i(t_0))}{D(q^r,q^l)}=\frac{D(p^i,q^i)}{D(q^r,q^l)}=0,$$
$$\frac{D(P^i(t_0),Q^i(t_0))}{D(p^r,q^l)}=\frac{D(p^i,q^i)}{D(p^r,q^l)}=\delta_{rl}.$$
Therefore, (\ref{Equation-QP0}) is fulfilled if and only if
\begin{equation}\label{Equation-QP1}
\sum_{i=1}^{n}d\frac{D(P^i(t),Q^i(t))}{D(p^r,p^l)}=\sum_{i=1}^{n}d\frac{D(P^i(t),Q^i(t))}{D(q^r,q^l)}=\sum_{i=1}^{n}d\frac{D(P^i(t),Q^i(t))}{D(p^r,q^l)}=0.
\end{equation}
Introduce the notation
$$
P_p^{ir}=\frac{\partial P^i}{\partial p^r}, \; \;  P_q^{ir}=\frac{\partial P^i}{\partial q^r}, \; \;
Q_p^{ir}=\frac{\partial Q^i}{\partial p^r},  \; \;  Q_q^{ir}=\frac{\partial P^i}{\partial q^r}.
$$
For a fixed $r$ , by calculating at $(P,Q)$ with $P=P(t)=(P^1(t;t_0,p,q),...,P^n(t;t_0,p,q))$ and $Q=Q(t)=(Q^1(t;t_0,p,q),...,Q^n(t;t_0,p,q))$ which is a solution to systems (\ref{Equation-q1} - \ref{Equation-p1}), we obtain $P_p^{ir}, Q_p^{ir}, i=1,...,n$, satisfy the following system of SDEs:
\begin{eqnarray}\label{Eqnarray-SDE1}
dP_p^{ir}&=& \sum_{j=1}^{n}(\frac{\partial f^i}{\partial p^j}P_p^{jr}+\frac{\partial f^i}{\partial q^j}Q_p^{jr})dt+ \sum_{k=1}^{d}\sum_{j=1}^{n}(\frac{\partial \sigma_k^i}{\partial p^j}P_p^{jr}+\frac{\partial \sigma_k^i}{\partial q^j}Q_p^{jr})\diamond dL^k,~ P_p^{ir}(t_0)=\delta_{ir},\notag\\
dQ_p^{ir}&=& \sum_{j=1}^{n}(\frac{\partial g^i}{\partial p^j}P_p^{jr}+\frac{\partial g^i}{\partial q^j}Q_p^{jr})dt+ \sum_{k=1}^{d}\sum_{j=1}^{n}(\frac{\partial \gamma_k^i}{\partial p^j}P_p^{jr}+\frac{\partial \gamma_k^i}{\partial q^j}Q_p^{jr})\diamond dL^k,~ Q_p^{ir}(t_0)=0,\notag\\
\end{eqnarray}
where
\begin{equation}\label{Equation-q1q}
f(Q,P)=\frac{\partial H}{\partial P}(Q,P), \; \;  \sigma_k (Q,P) =\frac{\partial H_k}{\partial P} (Q,P),
\end{equation}
\begin{equation}\label{Equation-p1p}
g(Q,P)=-\frac{\partial H}{\partial Q}(Q,P), \; \;   \gamma_k (Q,P) =-\frac{\partial H_k}{\partial Q} (Q,P),
\end{equation}
for $k = 1, . . . , m.$
\par
Then, we get
\begin{eqnarray}
dP_p^{ir}(t)Q_p^{il}(t)&=&\sum_{j=1}^{n}\left[
(\frac{\partial f^i}{\partial p^j}P_p^{jr}+\frac{\partial f^i}{\partial q^j}Q_p^{jr})Q_p^{il}+(\frac{\partial g^i}{\partial p^j}P_p^{jr}+\frac{\partial g^i}{\partial q^j}Q_p^{jr})P_p^{ir}
\right]dt\notag\\
&&+\sum_{k=1}^{m}\sum_{j=1}^{n}\left[
(\frac{\partial \sigma_k^i}{\partial p^j}P_p^{jr}+\frac{\partial \sigma_k^i}{\partial q^j}Q_p^{jr})Q_p^{il}+(\frac{\partial \gamma_k^i}{\partial p^j}P_p^{jl}+\frac{\partial \gamma_k^i}{\partial q^j}Q_p^{jl})P_p^{ir}
\right]\diamond dL^k.\notag
\end{eqnarray}
Similarly, we can also calculate $dP_p^{il}(t)Q_p^{ir}(t)$, then
\begin{eqnarray}
\sum_{i=1}^{n}d\frac{D(P^i(t),Q^i(t))}{D(p^r,p^l)}=\sum_{i=1}^{n}\left[
\sum_{j=1}^{n}\Xi_1dt+\sum_{k=1}^{m}\sum_{j=1}^{n}\Xi_2\diamond dL^k
\right],
\end{eqnarray}
where
\begin{eqnarray}
\Xi_1&=& \frac{\partial f^i}{\partial p^j}P_p^{jr}Q_p^{il}+\frac{\partial f^i}{\partial q^j}Q_p^{jr}Q_p^{il}+\frac{\partial g^i}{\partial p^j}P_p^{jl}P_p^{ir}+\frac{\partial g^i}{\partial q^j}Q_p^{jl}P_p^{ir}\notag\\
&&-\frac{\partial f^i}{\partial p^j}P_p^{jl}Q_p^{ir}-\frac{\partial f^i}{\partial q^j}Q_p^{jl}Q_p^{ir}-\frac{\partial g^i}{\partial p^j}P_p^{jr}P_p^{il}-\frac{\partial g^i}{\partial q^j}Q_p^{jr}P_p^{il},\notag\\
\Xi_2&=& \frac{\partial \sigma_k^i}{\partial p^j}P_p^{jr}Q_p^{il}+\frac{\partial \sigma_k^i}{\partial q^j}Q_p^{jr}Q_p^{il}+\frac{\partial \gamma_k^i}{\partial p^j}P_p^{jl}P_p^{ir}+\frac{\partial \gamma_k^i}{\partial q^j}Q_p^{jl}P_p^{ir}\notag\\
&&- \frac{\partial \sigma_k^i}{\partial p^j}P_p^{jl}Q_p^{ir}-\frac{\partial \sigma_k^i}{\partial q^j}Q_p^{jl}Q_p^{ir}-\frac{\partial \gamma_k^i}{\partial p^j}P_p^{jr}P_p^{il}-\frac{\partial \gamma_k^i}{\partial q^j}Q_p^{jr}P_p^{il}.\notag
\end{eqnarray}
It is not difficult to find out that, a sufficient condition of $\Xi_1=0$ is
\begin{equation}\label{Equation-sc1}
\frac{\partial f^i}{\partial p^j}=-\frac{\partial g^j}{\partial q^i},~ \frac{\partial f^i}{\partial q^j}=\frac{\partial f^j}{\partial q^i}, ~\frac{\partial g^i}{\partial p^j}=\frac{\partial g^j}{\partial p^i},
\end{equation}
and a sufficient condition of $\Xi_2=0$ is
\begin{equation}\label{Equation-sc2}
\frac{\partial \sigma_k^i}{\partial p^j}=-\frac{\partial \gamma_k^j}{\partial q^i},~ \frac{\partial \sigma_k^i}{\partial q^j}=\frac{\partial \sigma_k^j}{\partial q^i}, ~\frac{\partial \gamma_k^i}{\partial p^j}=\frac{\partial \gamma_k^j}{\partial p^i}.
\end{equation}
Noticing that relations (\ref{Equation-q1q} - \ref{Equation-p1p}) imply (\ref{Equation-sc1} - \ref{Equation-sc2}), we obtain ${\sum_{i=1}^{n}d\frac{D(P^i(t),Q^i(t))}{D(p^r,p^l)}=0}$. Similarly, we prove that the conditions  (\ref{Equation-q1q} - \ref{Equation-p1p}) ensure the other two terms of (\ref{Equation-QP1}) as well. This competes the proof.
\end{proof}


\begin{proof}[\textbf{Proof of Theorem \ref{theorem 3.3}.}] \emph{\textbf{}}
\\
We calculate  the derivatives of $\mathcal{S}$ with respect to $Q_0$ and $Q_1$:
\begin{eqnarray}
\frac{\partial\mathcal{S}}{\partial Q_0}
&=&\int_{t_0}^{t_1}\Big(
\frac{\partial L}{\partial Q} \frac{\partial Q}{\partial Q_0} + \frac{\partial L}{\partial \dot{Q}} \frac{\partial \dot{Q}}{\partial Q_0}+
\frac{\partial L}{\partial P} \frac{\partial P}{\partial Q_0} + \frac{\partial L}{\partial \dot{P}} \frac{\partial \dot{P}}{\partial Q_0}
\Big)dt \notag\\
&&-\sum_{k=1}^d\int_{t_0}^{t_1}\Big(\frac{\partial H_k}{\partial Q} \frac{\partial q}{\partial Q_0}+\frac{\partial H_k}{\partial P} \frac{\partial P}{\partial Q_0}\Big) \diamond d{L}^k(t) \notag\\
&=&\left[ \frac{\partial L}{\partial \dot{Q}} \frac{\partial Q}{\partial Q_0}\right]_{t_0}^{t_1}+\int_{t_0}^{t_1}\Big(
\frac{\partial L}{\partial Q}- \frac{d}{dt}\frac{\partial L}{\partial \dot{Q}} -\sum_{k=1}^d\frac{\partial H_k}{\partial Q} \diamond \dot{L}^k(t)
\Big)\frac{\partial Q}{\partial Q_0}dt \notag\\
&&+\left[ \frac{\partial L}{\partial \dot{P}} \frac{\partial P}{\partial Q_0}\right]_{t_0}^{t_1}+\int_{t_0}^{t_1}\Big(
\frac{\partial L}{\partial P}- \frac{d}{dt}\frac{\partial L}{\partial \dot{P}} -\sum_{k=1}^d\frac{\partial H_k}{\partial P} \diamond \dot{L}^k(t)
\Big)\frac{\partial p}{\partial Q_0}dt \notag\\
&=&-P_0^T,
\end{eqnarray}
where the last equality follows from the stochastic Lagrange equations (\ref{Eqnarray-Lagrange1_stoch} - \ref{Eqnarray-Lagrange2_stoch}) and the Legendre transform $P=\frac{\partial L}{\partial \dot{Q}} $.
\par
Similarly, we have
\begin{eqnarray}
\frac{\partial\mathcal{S}}{\partial Q_1}=-P_1^T.
\end{eqnarray}
Therefore,
\begin{equation}
d \mathcal{S}=-P_0^TdQ_0+P_1^TdQ_1.
\end{equation}
Moreover,
\begin{eqnarray}
&&dp_1\wedge dQ_1=d(\frac{\partial\mathcal{S}}{\partial Q_1})\wedge dQ_1=\frac{\partial^2\mathcal{S}}{\partial Q_1\partial Q_0}dQ_0\wedge dQ_1,\\
&&dp_0\wedge dQ_0=d(-\frac{\partial\mathcal{S}}{\partial Q_0})\wedge dQ_0=\frac{\partial^2\mathcal{S}}{\partial Q_0\partial Q_1}dQ_0\wedge dQ_1.
\end{eqnarray}
Smoothness of  $L$ and the $H_k (k = 1,...,d)$ in $\mathcal{S}$ ensures that $\frac{\partial^2\mathcal{S}}{\partial Q_1\partial Q_0}=\frac{\partial^2\mathcal{S}}{\partial Q_0\partial Q_1}$, which implies
\begin{equation}
dP_1\wedge dQ_1=dP_0\wedge dQ_0.
\end{equation}
The proof is thus complete.
\end{proof}

\begin{proof}[\textbf{Detailed Calculations of Example \ref{section 4.4}.}] \emph{\textbf{}}
\\
Recall that, by L\'evy-Khintchine formula \cite{Ap,Duan} the characteristic funtion for L\'evy motion  in $\mathbb{R}^d$ is
$$
\mathbb{E}e^{i \langle u, L_t \rangle}=e^{t\eta_0(u)},~u\in \mathbb{R}^d,
$$
where $\eta_0(u)=\int_{\mathbb{R}^d\setminus \{ 0 \}} \big[ e^{iu\cdot z} -1-i\mathbf{ 1}_{\{|z|<1\}}u\cdot z\big]\nu(dz)$ whose real part $\Re \eta_0 \leqslant 0$.
And the characteristic funtion for standard Brownian motion in $\mathbb{R}^d$ is
$$
\mathbb{E}e^{i \langle u, B_t \rangle}=e^{-\frac{1}{2}t\langle u, Iu \rangle}=e^{-\frac{1}{2}t|u|^2},~u\in \mathbb{R}^d.
$$
Therefore,
\begin{eqnarray}\label{Eqnarray-moment}
&&\mathbb{E}\Big[\frac{1}{t}\int_0^t g_i (q_s,p_s)ds \Big]
=\mathbb{E}\Big[\frac{1}{t} \int_0^t \tilde{g}_i (\theta_s,I_s)ds \Big]
=\mathbb{E}\Big[\frac{1}{t} \int_0^t \cos^2 (\langle u, B_s + L_s \rangle) \Big]ds\notag\\
&&
=\frac{1}{2t} \int_0^t \mathbb{E}\cos2 (\langle u, B_s + L_s \rangle) ds+\frac{1}{2}
=\frac{1}{2t} \int_0^t \Re \mathbb{E}e^{ i\langle 2u, B_s \rangle}\Re\mathbb{E}e^{i \langle 2u, L_s \rangle} ds+\frac{1}{2}\notag\\
&&=\frac{1}{2t} \int_0^t e^{-s(\frac{1}{2}|2u|^2-\Re\eta_0(2u))} ds+\frac{1}{2}
=\frac{1}{2t} \frac{1}{A}(1-e^{-At})+\frac{1}{2}.
\end{eqnarray}
Here $u=(1,1)^T\in \mathbb{R}^2$ and $A=\frac{1}{2}|2u|^2-\Re\eta_0(2u)>0$. Hence, as $t$ goes to $\infty$, the expectation is equal to $\frac{1}{2}$ eventually. Next, we calculate the secondary moment as following,
\begin{eqnarray}\label{Eqnarray-2moment}
&&\mathbb{E}\Big[\big|\frac{1}{t}\int_0^t g_i (q_s,p_s)ds\big|^2 \Big]
=\mathbb{E}\Big[\frac{1}{t^2} \big|\int_0^t \cos^2 (\langle u, B_s + L_s \rangle)ds \big|^2 \Big]\notag\\
&&=\frac{2}{t^2}\int_0^t\int_0^r \mathbb{E}\Big[\cos^2 (\langle u, B_s + L_s \rangle)\cos^2 (\langle u, B_r + L_r \rangle) \Big]ds dr \notag\\
&&=\frac{1}{4t^2}\int_0^t\int_0^r \mathbb{E}\Big[\Re e^{ i \langle 2u, (B_s + L_s)+ (B_r+ L_r) \rangle}+ \Re e^{i \langle 2u, (B_s + L_s)- (B_r+ L_r) \rangle}\notag\\
&&~~~~~~~~~~~~~~~~~~~~~~+2\Re e^{i \langle 2u, (B_s + L_s \rangle}+ 2\Re e^{i \langle 2u, (B_r+ L_r) \rangle}+2\Big] ds dr \notag\\
&&=\frac{1}{4t^2}\int_0^t\int_0^r \Big[ \mathbb{E}e^{ i \langle 4u, B_s\rangle} \mathbb{E}e^{ i \langle 2u, B_r-B_s\rangle}\mathbb{E}e^{ i \langle 4u, L_s\rangle} \mathbb{E}e^{ i \langle 2u, L_r-L_s\rangle}+\mathbb{E}e^{ i \langle 2u, B_r-B_s\rangle}\mathbb{E}e^{ i \langle 2u, L_r-L_s\rangle}\notag\\
&&~~~~~~~~~~~~~~~~~~~~~~+2\mathbb{E}e^{ i \langle 2u, B_s\rangle}\mathbb{E}e^{ i \langle 2u, L_s\rangle}+2 \mathbb{E}e^{ i \langle 2u, B_r\rangle}\mathbb{E}e^{ i \langle 2u, L_r\rangle}+2 \Big]ds dr \notag\\
&&=\frac{1}{4t^2}\int_0^t\int_0^r \Big[ e^{-(\frac{1}{2}|2u|^2-\Re\eta_0(2u))r-(\frac{1}{2}|4u|^2-\frac{1}{2}|2u|^2-\Re\eta_0(4u)+\Re\eta_0(2u))s}+e^{-(\frac{1}{2}|2u|^2+\Re\eta_0(2u))(r-s)}\notag\\
&&~~~~~~~~~~~~~~~~~~~~~~+2e^{-(\frac{1}{2}|2u|^2-\Re\eta_0(2u))s}+2 e^{-(\frac{1}{2}|2u|^2-\Re\eta_0(2u))r}+2 \Big]ds dr \notag\\
&&=\frac{1}{4t^2}\int_0^t\int_0^r \Big[ e^{-Ar-Bs}+e^{-C(r-s)}+2e^{-As}+2 e^{-Ar}+2 \Big]ds dr \notag\\
&&=\frac{1}{4t^2}\big[\frac{1}{A(A+B)}-\frac{1}{C^2}+(\frac{1}{C}+\frac{2}{A})t+t^2+\frac{e^{-(A+B)t}}{B(A+B)}- \frac{e^{-At}}{AB}+\frac{e^{-Ct}}{C^2}-\frac{2te^{-At}}{A}\big] .\notag\\
\end{eqnarray}
Here we used the stationary independent increments property of the Brownian motion and L\'evy motion. By Taylor expansion \cite[Page 40]{Sato} with $u=(1,1)^T$, we can find that $B=\frac{1}{2}|4u|^2-\frac{1}{2}|2u|^2-\Re\eta_0(4u)+\Re\eta_0(2u)>0$, $C=\frac{1}{2}|2u|^2+\Re\eta_0(2u)>0$, so we have $\mathbb{E}\Big[\big|\frac{1}{t}\int_0^t g_i (q_s,p_s)ds\big|^2\to \frac{1}{4}$ as $t\to \infty$.
Thus,
\begin{eqnarray}
&&\mathbb{E}\Big[\big|\frac{1}{t}\int_0^t g_i (q_s,p_s)ds -Q^g(h_i)\big|^2 \Big]
=\mathbb{E}\Big[\frac{1}{t^2} \big|\int_0^t \cos^2 (\langle u, B_s \rangle+\langle u, L_s \rangle)ds - \frac{1}{2} \big|^2 \Big]\notag\\
&&=\mathbb{E}\Big[\frac{1}{t^2} \big( \int_0^t \cos^2 (\langle u, B_s \rangle+\langle u, L_s \rangle)ds \big)^2\Big]-\mathbb{E}\Big[\frac{1}{t}\int_0^t \cos^2 (\langle u, B_s \rangle+\langle u, L_s \rangle)ds \Big]+\frac{1}{4}\notag\\
&&\to 0, ~ as ~t\to\infty.
\notag
\end{eqnarray}
Moreover, combining (\ref{Eqnarray-moment} - \ref{Eqnarray-2moment}) and taking the square root, the rate of convergence is of the order $\eta(t)=\frac{c}{\sqrt{t}}$ as $t\to\infty$ (c is a constant).
\end{proof}

\section*{Acknowledgements}
The authors are grateful to Qingshan Chen, Guan Huang, Haitao Xu, Guowei Yu, Jianlu Zhang, Lei Zhang and Yanjie Zhang for helpful discussions and comments. This work was partly supported by the NSFC grants 11531006 and  11771449.

\section*{References}
\bibliography{mybibfile}

\begin{thebibliography}{10}
\expandafter\ifx\csname url\endcsname\relax
  \def\url#1{\texttt{#1}}\fi
\expandafter\ifx\csname urlprefix\endcsname\relax\def\urlprefix{URL }\fi
\expandafter\ifx\csname href\endcsname\relax
  \def\href#1#2{#2} \def\path#1{#1}\fi

\bibitem{Ar}
V.~I. Arnol'd, Mathematical Methods of Classical Mechanics, Vol.~60, Springer
  Science \& Business Media, 2013.

\bibitem{Me}
J.~D. Meiss, Differential Dynamical Systems, Vol.~14, Siam, 2007.

\bibitem{Wig}
S.~Wiggins, Normally Hyperbolic Invariant Manifolds in Dynamical Systems, Vol.
  105, Springer Science \& Business Media, 2013.

\bibitem{Reza}
F.~Rezakhanlou, \href{https://math.berkeley.edu/~rezakhan/WKAM.pdf}{Hamiltonian
  {ODE}, Homogenization, and Symplectic Topology}, UC Berkeley.
\newline\urlprefix\url{https://math.berkeley.edu/~rezakhan/WKAM.pdf}

\bibitem{Poin}
H.~Poincar{\'e}, Les {M}{\'e}thodes Nouvelles de la M{\'e}canique C{\'e}leste:
  M{\'e}thodes de MM. Newcomb, Glyd{\'e}n, Lindstedt et Bohlin, Vol.~2,
  Gauthier-Villars it fils, 1893.

\bibitem{Tve}
F.~T. Tveter, Deriving the {H}amilton equations of motion for a nonconservative
  system using a variational principle, Journal of Mathematical Physics 39~(3)
  (1998) 1495--1500.

\bibitem{Fr}
M.~I. Freidlin, A.~D. Wentzell, Random Perturbations of Dynamical Systems, Vol.
  260, Springer Science \& Business Media, 2012.

\bibitem{Duan}
J.~Duan, An Introduction to Stochastic Dynamics, Cambridge University Press,
  2015.

\bibitem{Bi}
J.-M. Bismut, M{\'e}canique Al{\'e}atoire, Lecture Notes in Math, Springer,
  Berlin, 1981.

\bibitem{BMFM}
M.~Brin, M.~Freidlin, On stochastic behavior of perturbed hamiltonian systems,
  Ergodic Theory \& Dynamical Systems 20~(1) (2000) 55--76.

\bibitem{Mac}
R.~MacKay, Langevin equation for slow degrees of freedom of {H}amiltonian
  systems, in: Nonlinear Dynamics and Chaos: Advances and Perspectives,
  Springer, 2010, pp. 89--102.

\bibitem{Misa}
T.~Misawa, Conserved quantities and symmetries related to stochastic dynamical
  systems, Annals of the Institute of Statistical Mathematics 51~(4) (1999)
  779--802.

\bibitem{Wu}
L.~Wu, Large and moderate deviations and exponential convergence for stochastic
  damping {H}amiltonian systems, Stochastic Processes and their Applications
  91~(2) (2001) 205--238.

\bibitem{Zhu}
W.~Zhu, Z.~Huang, Stochastic stabilization of quasi-partially integrable
  {H}amiltonian systems by using {L}yapunov exponent, Nonlinear Dynamics 33~(2)
  (2003) 209--224.

\bibitem{Mi2}
G.~N. Milstein, Y.~M. Repin, M.~V. Tretyakov, Numerical methods for stochastic
  systems preserving symplectic structure, SIAM Journal on Numerical Analysis
  40~(4) (2002) 1583--1604.

\bibitem{Mi}
G.~N. Milstein, Y.~M. Repin, M.~V. Tretyakov, Symplectic integration of
  {H}amiltonian systems with additive noise, SIAM Journal on Numerical Analysis
  39~(6) (2002) 2066--2088.

\bibitem{Wang}
L.~Wang, J.~Hong, R.~Scherer, F.~Bai, Dynamics and variational integrators of
  stochastic {H}amiltonian systems., International Journal of Numerical
  Analysis \& Modeling 6~(4) (2009) 586--602.

\bibitem{Pav}
M.~Pavon, Hamilton's principle in stochastic mechanics, Journal of Mathematical
  Physics 36~(12) (1995) 6774--6800.

\bibitem{Li}
X.-M. Li, An averaging principle for a completely integrable stochastic
  {H}amiltonian system, Nonlinearity 21~(4) (2008) 803--822.

\bibitem{Cr}
J.~Cresson, S.~Darses, Stochastic embedding of dynamical systems, Journal of
  Mathematical Physics 48~(7) (2007) 365--373.

\bibitem{Gi}
D.~Givon, R.~Kupferman, A.~Stuart, Extracting macroscopic dynamics: Model
  problems and algorithms, Nonlinearity 17~(6) (2004) R55--127.

\bibitem{Ap}
D.~Applebaum, {L}{\'e}vy Processes and Stochastic Calculus, Cambridge
  University Press, 2009.

\bibitem{Sato}
K.-I. Sato, L{\'e}vy Processes and Infinitely Divisible Distributions,
  Cambridge University Press, 1999.

\bibitem{GarL}
L.~Garnett, Foliations, the ergodic theorem and brownian motion, Journal of
  Functional Analysis 51~(3) (1983) 285--311.

\bibitem{Br}
S.~Albeverio, Z.~Brze\'{z}niak, J.-L. Wu, Existence of global solutions and
  invariant measures for stochastic differential equations driven by poisson
  type noise with non-{L}ipschitz coefficients, Journal of Mathematical
  Analysis and Applications 371~(1) (2010) 309--322.

\bibitem{Al}
S.~Albeverio, B.~R{\"u}diger, J.-L. Wu, Invariant measures and symmetry
  property of {L}{\'e}vy type operators, Potential Analysis 13~(2) (2000)
  147--168.

\bibitem{Hog}
M.~H{\"o}gele, P.~Ruffino, Averaging along foliated {L}{\'e}vy diffusions,
  Nonlinear Analysis: Theory, Methods \& Applications 112 (2015) 1--14.

\bibitem{GGI}
I.~I.~G. Gargate, P.~R. Ruffino, An averaging principle for diffusions in
  foliated spaces, The Annals of Probability 44~(1) (2016) 567--588.

\bibitem{Hog2}
M.~A. H{\"o}gele, P.~H. da~Costa, Strong averaging along foliated {L}{\'e}vy
  diffusions with heavy tails on compact leaves, Potential Analysis 47~(3)
  (2017) 277--311.

\bibitem{Xu}
Y.~Xu, J.~Duan, W.~Xu, An averaging principle for stochastic dynamical systems
  with {L}{\'e}vy noise, Physica D: Nonlinear Phenomena 240~(17) (2011)
  1395--1401.

\bibitem{Bao}
J.~Bao, G.~Yin, C.~Yuan, Two-time-scale stochastic partial differential
  equations driven by $\alpha $-stable noises: Averaging principles, Bernoulli
  23~(1) (2017) 645--669.

\bibitem{Ku}
H.~Kunita, Stochastic differential equations based on {L}{\'e}vy processes and
  stochastic flows of diffeomorphisms, in: Real and Stochastic Analysis,
  Springer, 2004, pp. 305--373.

\bibitem{KPP}
T.~G. Kurtz, {\'E}.~Pardoux, P.~Protter, Stratonovich stochastic differential
  equations driven by general semimartingales, in: Annales de l'IHP
  Probabilit{\'e}s et statistiques, Vol.~31, 1995, pp. 351--377.

\bibitem{Mar}
S.~I. Marcus, Modeling and approximation of stochastic differential equations
  driven by semimartingales, Stochastics: An International Journal of
  Probability and Stochastic Processes 4~(3) (1981) 223--245.

\bibitem{El}
K.~D. Elworthy, Stochastic Differential Equations on Manifolds, Vol.~70,
  Cambridge University Press, 1982.

\bibitem{Protter}
P.~E. Protter, Stochastic Integration and Differential Equations, Springer,
  2005.

\bibitem{AM}
R.~Abraham, J.~E. Marsden, Foundations of Mechanics, Benjamin/Cummings, 1978.

\bibitem{Xi}
F.~Xi, C.~Zhu, Jump type stochastic differential equations with non-{L}ipschitz
  coefficients: Non-confluence, {F}eller and strong {F}eller properties, and
  exponential ergodicity, Journal of Differential Equations (2018) in press.

\bibitem{Fox}
A.~M. Fox, R.~de~la Llave, Barriers to transport and mixing in
  volume-preserving maps with nonzero flux, Physica D: Nonlinear Phenomena 295
  (2015) 1--10.

\bibitem{Ar2}
V.~I. Arnold, V.~V. Kozlov, A.~I. Neishtadt, Mathematical {A}spects of
  Classical and Celestial Mechanics, Vol.~3, Springer Science \& Business
  Media, 2007.

\bibitem{KreU}
U.~Krengel, On the speed of convergence in the ergodic theorem, Monatshefte
  f{\"u}r Mathematik 86~(1) (1978) 3--6.

\bibitem{KSPK}
S.~Kakutani, K.~Petersen, The speed of convergence in the ergodic theorem,
  Monatshefte f{\"u}r Mathematik 91~(1) (1981) 11--18.

\bibitem{Kulik}
A.~M. Kulik, Exponential ergodicity of the solutions to {SDE}'€™s with a jump
  noise, Stochastic Processes and their Applications 119~(2) (2009) 602--632.

\bibitem{PBG}
B.~G. Pachpatte, Inequalities for Differential and Integral Equations, Academic
  Press, 1998.

\end{thebibliography}

\end{document}